\documentclass[12pt]{amsart}
\usepackage{amsmath,amsthm,amssymb,amsfonts,enumerate,color}
\usepackage{mathrsfs}
 \usepackage{amsmath,amstext,amsthm,amssymb,amsxtra}
 \usepackage{txfonts} 
   \usepackage[colorlinks, citecolor=blue,pagebackref,hypertexnames=false]{hyperref}
 \allowdisplaybreaks
 \usepackage{pgf,tikz}
\usepackage{relsize}

\begin{document}
\newcommand\RR{\mathbb{R}}
\newcommand{\la}{\lambda}
\def\RN {\mathbb{R}^n}
\newcommand{\norm}[1]{\left\Vert#1\right\Vert}
\newcommand{\abs}[1]{\left\vert#1\right\vert}
\newcommand{\set}[1]{\left\{#1\right\}}
\newcommand{\Real}{\mathbb{R}}
\newcommand{\supp}{\operatorname{supp}}
\newcommand{\card}{\operatorname{card}}
\renewcommand{\L}{\mathcal{L}}
\renewcommand{\P}{\mathcal{P}}
\newcommand{\T}{\mathcal{T}}
\newcommand{\A}{\mathbb{A}}
\newcommand{\K}{\mathcal{K}}
\renewcommand{\S}{\mathcal{S}}
\newcommand{\blue}[1]{\textcolor{blue}{#1}}
\newcommand{\red}[1]{\textcolor{red}{#1}}
\newcommand{\Id}{\operatorname{I}}
\newcommand\wrt{\,{\rm d}}
\def\SL{\sqrt {L}}

\newcommand{\mar}[1]{{\marginpar{\sffamily{\scriptsize
        #1}}}}
\newcommand{\li}[1]{{\mar{LY:#1}}}
\newcommand{\el}[1]{{\mar{EM:#1}}}
\newcommand{\as}[1]{{\mar{AS:#1}}}

\newcommand\CC{\mathbb{C}}
\newcommand\NN{\mathbb{N}}
\newcommand\ZZ{\mathbb{Z}}

\renewcommand\Re{\operatorname{Re}}
\renewcommand\Im{\operatorname{Im}}

\newcommand{\mc}{\mathcal}
\newcommand\D{\mathcal{D}}

\newcommand{\al}{\alpha}
\newcommand{\nf}{\infty}

\newcommand{\comment}[1]{\vskip.3cm
	\fbox{%
		\color{red}
		\parbox{0.93\linewidth}{\footnotesize #1}}
	\vskip.3cm}

\newcommand{\disappear}[1]

\newtheorem{thm}{Theorem}[section]
\newtheorem{prop}[thm]{Proposition}
\newtheorem{cor}[thm]{Corollary}
\newtheorem{lem}[thm]{Lemma}
\newtheorem{lemma}[thm]{Lemma}
\newtheorem{exams}[thm]{Examples}
\theoremstyle{definition}
\newtheorem{defn}[thm]{Definition}
\newtheorem{rem}[thm]{Remark}

\numberwithin{equation}{section}
\newcommand{\chg}[1]{{\color{red}{#1}}}
\newcommand{\note}[1]{{\color{green}{#1}}}
\newcommand{\later}[1]{{\color{blue}{#1}}}
\newcommand{\bchi}{\mathlarger{\chi}}

\title[ ]
{ Bounds on the maximal Bochner-Riesz means  \\[2pt]  for  elliptic operators}

\author{Peng Chen}
\author{Sanghyuk Lee}
\author{Adam Sikora}
\author{Lixin Yan}
\address{Peng Chen, Department of Mathematics, Sun Yat-sen
University, Guangzhou, 510275, P.R. China}
\email{chenpeng3@mail.sysu.edu.cn}
\address{Sanghyuek Lee,  School of Mathematical Sciences, Seoul national University, Seoul 151-742, Repulic of Korea}
\email{shklee@snu.ac.kr}
\address {Adam Sikora, Department of Mathematics, Macquarie University, NSW 2109, Australia}
\email{sikora@maths.mq.edu.au}
\address{Lixin Yan, Department of Mathematics, Sun Yat-sen   University,
Guangzhou, 510275, P.R. China}
\email{mcsylx@mail.sysu.edu.cn}
\date{\today}
\subjclass[2000]{42B15, 42B25,   47F05.}
\keywords{Maximal Bochner-Riesz means, non-negative self-adjoint operators,    finite speed
propagation property, elliptic type estimates, restriction type conditions.}

\begin{abstract}
 We investigate $L^p$  boundedness of  the maximal Bochner-Riesz means  for self-adjoint operators of elliptic type.
 Assuming the finite speed  of  propagation for the associated wave operator,  from the restriction type estimates
 we establish the sharp $L^p$ boundedness of  the maximal Bochner-Riesz means  for the elliptic operators.  As applications,
 we obtain the sharp $L^p$ maximal bounds  for  the Schr\"odinger operators  on asymptotically conic manifolds,  the harmonic oscillator
 and its perturbations or  elliptic operators on compact manifolds.
 \end{abstract}

\maketitle



\section{Introduction } \setcounter{equation}{0}
Convergence of the Bochner-Riesz means and boundedness of the  associated maximal operators  on
Lebesgue $L^p$ spaces are among  the most classical problems in harmonic analysis.  The study on the Bochner-Riesz means
can be seen as an attempt to justify the Fourier inversion. We begin with recalling the Bochner-Riesz means on
$\RN$ which  are  defined by,  for $\alpha\ge 0$ and $R>0$,
\begin{eqnarray}\label{ee0}
\widehat{S^{\alpha}_Rf}(\xi)
=\left(1-{|\xi|^2\over R^2}\right)_+^{\alpha} \widehat{f}(\xi),  \quad \forall{\xi \in \RN}.
\end{eqnarray}
Here $(x)_+=\max\{0,x\}$ for $x\in \mathbb R$ and $\widehat{f}\,$ denotes the Fourier  transform
of $f$.
The associated  maximal function which is called `maximal Bochner-Riesz  operator' is given by
\begin{eqnarray}\label{ee00}
 S_{\ast}^{\alpha} f(x)=  \sup_{R>0} |S^{\alpha}_Rf(x) |.
\end{eqnarray}
 The problem of characterizing  the  optimal range of $\alpha$ for which $S^\alpha$ (and $S^\alpha_\ast$) is  bounded  on $L^p(\RN)$
 is known as  the  Bochner-Riesz  (and maximal Bochner-Riesz) conjecture.
It  has been  conjectured that, for $1\le p \le \infty $ and $p\neq 2$, $S_{R}^\alpha $ is bounded on $L^p(\mathbb R^n)$ if and only if
\begin{equation}
\label{sharp}
\alpha> \alpha(p)=\max \left\{ n \left|{1\over p}-{1\over 2} \right|-{1\over 2}, 0 \right\}.
\end{equation}
We refer the reader to  \cite{DY},  Stein's monograph  \cite[Chapter IX]{St2}  and Tao \cite{Ta3} for historical background and more on
the Bochner-Riesz conjecture.
It was shown by Herz that for a given $p$  the above condition on  $\alpha$  is
necessary, see \cite{He}.  Carleson and Sj\"olin \cite{CS} proved the conjecture when $n=2$.  Afterward substantial progress has bee
made \cite{tvv, Lee, BoGu, ghi}, but the conjecture still remains open for $n\ge 3$.

 Concerning the $L^p$ boundedness of $S^\alpha_\ast$,
for $p\geq 2$ it is natural to expect that
$S_{\ast}^\alpha$ is
bounded on $L^p$ on the same range where $S_{R}^\alpha$ is bounded, see e.g. \cite{Lee, LRS1}.  This was shown to be true by   Carbery \cite{Ca} when $n=2$.
In dimensions greater
than two partial results are known.    Christ \cite{C1} showed that    $S_{\ast}^{\alpha} $ is bounded on $L^p$ if
$p\geq {2(n+1)/(n-1)}$ and $\alpha> \alpha(p)$,
and the range of $p$ was extended by the second named author  to the range  $p>{2(n+2)/n}$ in \cite{Lee} and see \cite{Lee2} for  the most recent progress.
In this paper  we focus on the case $p\ge 2$ but it should be mentioned that, for $p<2$,  the range of $\alpha$ where
$S_{\ast}^{\alpha}$ is bounded on $L^p$ is different
from that of   $S_R^{\alpha}$.  Tao \cite{Ta1} showed that
the additional restriction $\alpha\geq (2n-1)/(2p) - n/2$ is necessary. Besides, when $n=2$ he obtained an improved estimate over
 the classical result \cite{Ta2}.


\subsubsection*{Bochner-Riesz means for elliptic operators } Since the Bochner-Riesz means are radial Fourier multipliers,
they can be defined   in terms of the spectral resolution of
the standard Laplace operator
$\Delta=\sum_{i=1}^n\partial_{x_i}^2$. This point of view naturally allows us to extend the  Bochner-Riesz means
and the  maximal   Bochner-Riesz operator to arbitrary positive
self-adjoint operator. For this purpose suppose  that $(X,d,\mu) $ is a metric measure space
 with a distance $d$ and a measure  $\mu$,  and that $L$
is a non-negative self-adjoint operator acting on the
space $L^2(X)$.   Such an operator admits a spectral
resolution
\begin{eqnarray*}
L=\int_0^{\infty} \lambda dE_L(\lambda).
\end{eqnarray*}
Now, the Bochner-Riesz mean of order $\alpha\geq 0$   can be defined by
 \begin{equation}\label{e1.22}
 S^{\alpha}_R(L) f(x)= \left( \int_0^{R^2} \left(1-\frac{\lambda}{R^2}\right)^{\alpha}dE_L(\lambda) f \right)(x), \ \ \ \ x\in X
   \end{equation}
  and the associated   maximal operator is given by
  \begin{eqnarray}\label{e1.2}
  S_{\ast}^{\alpha}(L) f(x)= \sup_{R>0} |S^{\alpha}_R(L) f(x) |.
  \end{eqnarray}
If we set $L=-\Delta$, the operators $S^{\alpha}_R(-\Delta)$ and
$S_{\ast}^{\alpha}(-\Delta)$ coincide with the classical  $S_R$ and $S_{\ast}^{\alpha}$, respectively. In this paper we aim to
investigate $L^p$-boundedness of the maximal Bochner-Riesz given by a  certain class of self-adjoint operators.

\subsubsection*{Restriction estimates} The celebrated Stein-Tomas restriction estimate to the sphere played an important role in
the development  of Bochner-Riesz problem (see \cite{St2}). This estimate can be reformulated in terms of spectral decomposition
of the standard Laplace operator. Indeed,  for $\lambda>0$ let
$R_\lambda$  be the restriction operator given by    $R_\lambda(f)(\omega) =\hat{f}({\lambda} \omega),$
where $\omega\in   {\bf S}^{n-1}$ (the unit sphere).
Then
$$
d E_{\sqrt{-\Delta}}(\lambda) =(2\pi)^{-n} \lambda^{n-1}R_\lambda^*R_\lambda.
$$
Thus, putting $L=-\Delta$, the  {  Stein-Tomas theorem} (\cite[p. 386]{St2}) 
is equivalent to the estimate
\begin{equation}
\label{e1.4}
\|dE_{\sqrt L}(\lambda)\|_{p\to p'}\le C\ \lambda^{{n}(\frac{1}{p}-\frac{1}{p'})-1}, \ \ \lambda>0\
 \end{equation}
 for $1\le p\le 2(n+1)/(n+3)$. In \cite{GHS} Guillarmou,
  Hassell and the third named author showed that the estimate \eqref{e1.4}
 remains valid for the Schr\"odinger type operators on asymptotically conic manifolds. It is easy to check that \eqref{e1.4} is equivalent to the
 following estimate:
 \begin{equation*}\label{rp}
 \tag{${\rm R_p}$}
 \big\|F(\!\SL\,\,) \big\|_{p\to 2} \leq C R^{n\left({\frac 1p}-{\frac 12}\right)} \big\| \delta_R F  \big\|_{2}
 \end{equation*}
  for any $R>0$ and all Borel functions $F$ supported   in $ [0,R],$
where  the dilation $\delta_RF$ is defined by
 {$\delta_RF(x)=F(Rx)$} (see  \cite[Proposition I.4]{COSY} ).

Observation regarding relation  between restriction estimate and the sharp $L^p$-boundedness (the boundedness of $S^\alpha_R$ in $L^p$
for $\alpha$ satisfying \eqref{sharp}) of the Bochner-Riesz means goes back to as far as  Stein \cite{fe1} (and also see \cite{St2}).
The argument in \cite{fe1} and the Stein-Tomas restriction estimate give the sharp $L^p$ estimates for
$S_R^\alpha(-\Delta)$ for $p$ satisfying  $\max(p,p')\ge 2(n+1)/(n-1)$.  Likewise, it is natural to suspect if there is a similar
connection between \eqref{rp}  and  the sharp $L^p$ bound for  $S^{\alpha}_R(L)$ when $L$ is a general elliptic operator.
This question was explored in \cite{COSY}. In fact, it was  shown in  \cite[Corollary I.6]{COSY}  that if the operator
$L$ satisfies the finite speed of propagation property
and the condition \eqref{rp}, then the Bochner-Riesz means are bounded on $L^p(X)$ spaces for $p$ on the range where \eqref{rp}
holds  if   ${\alpha}>\max(0, n|1/p-1/2|-1/2)$.


Our first result is the maximal generalization of the aforementioned result in \cite{COSY}.

\noindent{\bf Theorem A.}  {\it  Let  { $B(x,r)=\{y\in X: d(x,y)<r\}$} and {$V(x,r)= \mu\big( B(x,r)\big)$}. Suppose  that
\begin{equation}
\label{eq1.1}
C^{-1}r^n \leq V(x, r)\leq C r^n 
\end{equation}
holds for all $x\in X$, and $L$  satisfies the finite speed of propagation property (see, Definition \ref{FSP})
and the condition  {\color{red}${\rm ( R_{p_0}})$} for some  $1\leq p_0 <2$.
    Then the operator $S_{\ast}^{\alpha}(L)$ is bounded on $L^p(X)$ whenever
 \begin{eqnarray}\label{eww}
  2\leq p< p_0',  \ \ \ {\rm and }\ \ \  \alpha> \alpha(p_0)=
  \max\left\{ n\left({1\over p_0}-{\frac 1 2}\right)- {\frac 1 2}, \, 0 \right\}.
  \end{eqnarray}
 As a consequence, if  $f\in L^p(X)$, then  for $p$ and $\alpha$ satisfying \eqref{eww},
  $$
  \lim\limits_{R\to \infty}S_{R}^{\alpha}(L)f(x)=f(x), \ \ \ a.e.
  $$
}

Later, we will see that the condition \eqref{eq1.1} can be replaced
by the doubling condition  \eqref{eq2.2}.

\subsubsection*{Cluster estimates.} It is not difficult to see that  the  condition   \eqref{rp} implies that
  the set of point spectrum of $L$ is empty.   Indeed, one has, for   $0\leq a< R$, 
$
\|1\!\!1_{\{a \}  }(\sqrt{L}\,) \|_{p\to 2}
\leq C R^{n({1\over p}-{1\over 2})} \|1\!\!1_{\{a \} }(R\cdot)\|_{2} =0,
$
and thus $1\!\!1_{\{a\} }(\sqrt{L}\,)=0$. Since  $\sigma(L)\subseteq [0, \infty)$, it is clear that the point spectrum of $L$
is empty.
In particular, \eqref{rp}  does  not hold  for elliptic
operators on compact manifolds or for the harmonic oscillator.
 In order to treat these cases as well we need to modify the estimate \eqref{rp} as follows:
 For a fixed natural number $\kappa$ and for all $N\in \NN$ and all  even
  Borel functions $F$  supported in $[-N, N]$,
$$
\big\|F(\!\SL\,\,) \big\|_{p\to 2} \leq
 CN^{n({\frac 1p}-{\frac 12})}\| \delta_N F  \|_{N^\kappa,\, 2},
\leqno{\rm (SC^{\kappa}_{p})}
$$
 where
 \begin{equation}
\label{norm-def}
\|F\|_{N,2}:=\left({1\over 2N}\sum_{\ell=1-N}^{N} \sup_{\lambda\in
 [{\ell-1\over N}, {\ell\over N})} |F(\lambda)|^2\right)^{1/2}
\end{equation}
for $F$ with $\supp F \subset [-1, 1]$. The  norm $\|F\|_{N,2}$ already appeared in \cite{CowS, DOS} in the study of
spectral multipliers, see also \cite{COSY}.

As shown in  \cite[Proposition I.14]{COSY},
the condition  ${\rm (SC^{1}_{p})}$ is equivalent to the following $(p, p')$ spectral cluster estimate
${\rm (S_p)}$
introduced by Sogge (see \cite{Sog1, Sog3, Sog4}): For all $\lambda \geq 0,$
 $$
\big\|E_{\sqrt{L}}[\lambda, \lambda+1)\big\|_{p\to p'} \leq
C (1+\lambda)^{n({1\over p}-{1\over p'})-1}. \leqno{\rm (S_p)}
$$
In this context we shall prove the following result.

\noindent{\bf Theorem B.} {\it
Suppose  that the condition
 \begin{equation}\label{eq2a}
 \mu(X)<\infty \quad \mbox{and} \quad C^{-1} \min(r^n, 1) \le V(x, r)\leq C \min(r^n, 1)
 \end{equation}
is valid   for all $x\in X$ and $r>0$. And
 suppose that  the operator $L$ satisfies  the finite speed of propagation property  (see, Definition \ref{FSP})
and  the condition  ${\rm (SC^{1}_{p_0})}$.
Then the operator $S_{\ast}^{\alpha}(L)$ is bounded on $L^p(X)$ whenever \eqref{eww} is satisfied.
%
  As a consequence, if $f\in L^p(X)$, then  for $p$ and $\alpha$  satisfying \eqref{eww}
  $$
  \lim\limits_{R\to \infty}S_{R}^{\alpha}(L)f(x)=f(x), \ \ \ a.e.
  $$
}

 We now consider the case $\mu(X)=\infty$ with the property \eqref{eq1.1}. Motivated by the harmonic oscillator $L=-\Delta+|x|^2$
we obtain the following variant of {  Theorem B}.

\noindent{\bf Theorem C.} {\it
	Suppose  that condition \eqref{eq1.1} holds, and  the operator $L$ satisfies  the finite speed of propagation property
	and the condition  ${\rm (SC^{\kappa}_{p_0})}$ for  some
	   $1\leq p_0<2$ and  some positive integer
	$\kappa$. In addition,   we assume that there exists $\nu \ge 0$  such that
	\begin{eqnarray}\label{e1.500}
	\|(1+L)^{-\gamma/2}\|_{{p'_0}\to 2}\leq C, \  \  \gamma=n(\kappa-1)(1/p_0-1/2)+\kappa\nu.
	\end{eqnarray} 
	Then the operator $S_{\ast}^{\alpha}(L)$ is bounded on $L^p(X)$ whenever 
	\begin{eqnarray}\label{ewww2} 	2\leq p< p_0',  \ \ \ {\rm and }\ \ \  \alpha> \nu+ 	
	\max\left\{ n\left({\frac {1}{p_0}}-{\frac 12}\right)- {\frac 12}, \, 0 \right\}. 	\end{eqnarray}
	As a consequence, if $f\in L^p(X)$, then for  $p$ and $\alpha$ satisfying \eqref{ewww2},
	$$
	\lim\limits_{R\to \infty}S_{R}^{\alpha}(L)f(x)=f(x), \ \ \ a.e.
	$$
	}


  We shall show  that in dimension $n\geq 2$,
 \eqref{e1.500} holds with $\kappa=2$ and each $\nu>0$ for  the harmonic oscillator  $L=-\Delta +|x|^2$
and $L=-\Delta +V(x)$  with the potential $V$ satisfying \eqref{eq111.01} below. The restriction estimates ${\rm (SC^{2}_{p})}$
for those operator were obtained by  Kardzhov \cite{Kar}, Thangavelu \cite{Th4},  Koch and Tataru \cite{KoT}.  Combining these
estimates with Theorem C, we are able to obtain the sharp $L^p$ bounds for the associated maximal Bochner-Riesz  operators.  See Section 6.3.

In order to prove  Theorems A, B, and C, we make use of the square function which has been utilized to control
the maximal Bochner-Riesz operators (see  \cite{St, Ca, C1, Lee}). The square function estimates in Proposition \ref{prop4.1} and
Proposition \ref{prop5.1} also have other applications. In particular, those estimates can be used to deduce smoothing properties
for the Schr\"odinger and the wave equations and also spectral multiplier theorems of H\"ormander-Mihlin type, see \cite{LRS1, LRS2}
for such implications  when $L=-\Delta$.  However,  unlike the classical case $L=-\Delta$,  for the general elliptic operators we
 don't have the  typical properties of Fourier multipliers such as translation and scaling invariances.  Also, the associated  heat
  kernels are not necessarily smooth. This requires to refine
  the classical argument in various aspects. In particular we will use
   a new variant of Calder\'on--Zygmund technique
 for the square functions, see for example  \cite{Au, AM}.

 Roughly speaking, we show that  the estimate \eqref{rp} (equivalently \eqref{e1.4}) or its variant  implies the $L^p$ boundedness of
 the maximal Bochner-Riesz operators assuming the finite speed of propagation   property. Main advantage of this  approach is that we can
 handle large class of elliptic operators. Since the restriction type estimates are better understood now,  it is possible  to extend
 part of this argument to general setting of the homogeneous spaces, and also to include operators such as harmonic oscillator or
 operators acting on compact manifolds.

 The   Bochner-Riesz means operator   for various classes of self-adjoint operators have been extensively studied  (see
\cite{COSY, DOS, Heb,  Ho2, Kar,  Ma,  Se, SYY2, Sog1, Sog4, Th1, Th3, Th4}
and references therein). However,
as far as the authors are aware, there is no result that proves,   on the range of $p$ up to that of restriction type estimate,
the sharp $L^p$ boundedness of the maximal
Bochner-Riesz operator other than  the standard Laplacian  and Fourier multipliers (see \cite{BoGu,
Ca,  CS, C1, F2,  Lee, Lee2, LRS1, LRS2, Se2, SW}).

\emph{Organization of the paper.} In Section 2 we provide some prerequisites, which we need later, mostly on the restriction type
estimate and the finite speed  of  propagation property.  In Section 3 we consider the maximal bounds under less restrictive assumptions
 which includes more general elliptic operators though they don't give the sharp bounds. The proof of Theorem~A will be given in Section 4.
   The proof of Theorems~B and ~C will be given in Section 5. In Section \ref{sec6} we discuss some examples of applications of Theorems A, B, C which
include the harmonic oscillator and its perturbation, Schr\"odinger operators  on asymptotically conic manifolds,
elliptic operators on compact manifolds and the radial part of the standard Laplace operator.


{\bf List of notation.}
\\
$ \bullet$ $(X,d,\mu) $ denotes  a metric measure space
 with a distance $d$ and a measure  $\mu$.
 \\
$ \bullet$ $L$
is a non-negative self-adjoint operator acting on the
space $L^2(X).$
\\
$ \bullet$ For $x\in X$ and $r>0$, $B(x,r)=\{y\in X: d(x,y)<r\}$  and {$V(x,r)= \mu\big( B(x,r)\big).$
\\
$ \bullet$  $\delta_RF$ is defined by
$\delta_RF(x)=F(Rx)$ for $R>0$ and Borel function $F$ supported  on $ [0,R].$
 \\
 $\bullet$
 $[t]$ denotes the integer part of $t$
for any positive real number $t$.
\\
 $\bullet$ $\NN$ is the set of positive integers.
\\
 $\bullet$ For $p\in [1,\infty]$, $p'={p}/{(p-1)}$.
 \\
 $\bullet$  For $1\le p\le\infty$ and  $f\in L^p(X,{\rm d}\mu)$,  $\|f\|_p=\|f\|_{L^p(X,{\rm d}\mu)}.$
 \\
 $\bullet$
  $\langle \cdot,\cdot \rangle$ denotes
the scalar product of $L^2(X, {\rm d}\mu)$.
\\
 $\bullet$ For $1\le p, \, q\le+\infty$,  $\|T\|_{p\to q} $ denotes the  operator norm of $T$
 from $ L^p(X, {\rm d}\mu)$ to $L^q(X, {\rm d}\mu)$.
 \\
 $\bullet$ If $T$ is  given by $Tf(x)=\int K(x,y) f(y) d\mu(y)$,  we denote by  $K_T$ the kernel of $T$.
 \\
 $\bullet$
 Given a  subset $E\subseteq X$, we  denote by  $\bchi_E$   the characteristic
function of   $E$.
\\
 $\bullet$
 For  $1\leq r <\infty$, $\mathfrak M_r$ denote the uncentered  $r$-th maximal operator over balls in $X$, that is
 \begin{equation*}
\mathfrak  M_rf(x)=\sup_{x\in B} \left({1\over \mu(B) }\int_{B}
|f(y)|^rd\mu(y)\right)^{1/r}.
\end{equation*}
 For simplicity we denote by $\mathfrak M$ the Hardy-Littlewood maximal function $\mathfrak M_1$.

\section{Preliminaries}\label{sec2}
\setcounter{equation}{0}


We say that $(X, d, \mu)$ satisfies
 the doubling property (see Chapter 3, \cite{CW})
if there  exists a constant $C>0$ such that
\begin{eqnarray}
V(x,2r)\leq C V(x, r)\quad \forall\,r>0,\,x\in X.  \label{eq2.1}
\end{eqnarray}
If this is the case, there exist  $C, n$ such that for $\lambda\geq 1$ and $x\in X$
\begin{equation}
V(x, \lambda r)\leq C\lambda^n V(x,r). \label{eq2.2}
\end{equation}
In the Euclidean space with Lebesgue measure, $n$ corresponds to the dimension of the space.
Observe that if $X$ satisfies (\ref{eq2.1}) and has finite measure then it has finite diameter.
Therefore, if $\mu(X)$ is finite, then we may assume that $X=B(x_0, 1)$ for some $x_0\in X$.


\noindent{\bf 2.1.
Finite speed of propagation  property and elliptic type estimates.} \
To formulate the finite speed of propagation  property for the wave equation corresponding to an operator $L$, we
set
\begin{equation*}
\D_r=\{ (x,\, y)\in X\times X: {d}(x,\, y) \le r \}.
\end{equation*}
Given an  operator $T$ from $L^p(X)$ to $L^q(X)$,
we write
\begin{equation}\label{eq2.4}
\supp  K_{T} \subseteq \D_r
\end{equation}
 if $\langle T f_1, f_2 \rangle = 0$ whenever $f_1 \in L^p(B(x_1,r_1))$, $f_1 \in L^{q'}( B(x_2,r_2))$ with $r_1+r_2+r < {d}(x_1, x_2)$.
 Note that  if $T$ is an
integral operator with a  {kernel $K_T$}, then (\ref{eq2.4}) coincides
with the  standard meaning of $\supp  K_{T}  \subseteq \D_r$,
 that is $K_T(x, \, y)=0$ for all $(x, \, y) \notin \D_r$.

\begin{defn}\label{FSP}
Given  a non-negative self-adjoint operator $L$
on $L^2({X})$, we say that $L$ satisfies the finite speed of
propagation property if
 \begin{equation*}\label{FS}
  \tag{FS}
\supp  K_{\cos(t\SL\,\,)} \subseteq \D_t, \quad \forall t> 0\,.
\end{equation*}
\end{defn}

Property \eqref{FS} holds for most of second order self-adjoint operators  and  is equivalent to celebrated Davies-Gaffney
estimates, see for example  \cite{CouS} and  \cite{S2}.

\begin{lemma}\label{le2.2}
Assume that $L$ satisfies the   property  {\rm (FS)} and that $F$ is an even bounded Borel function with Fourier
transform  $\hat{F} \in L^1(\RR)$ and that
$\supp \hat{F} \subseteq [-r, r]$.
Then
$$
\supp K_{F(\!\SL\,\,)} \subseteq \D_r.
$$
\end{lemma}

\begin{proof}
If $F$ is an even function, then by the Fourier inversion formula,
$$
F(\!\SL\,\,) =\frac{1}{2\pi}\int_{-\infty}^{+\infty}
  \hat{F}(t) \cos(t\SL\,\,) \;dt.
$$
But $\supp\hat{F} \subseteq [-r,r]$,
and the lemma follows then from {\rm (FS)}.
\end{proof}

Since our discussion covers general elliptic operators, we need  some related estimates which are slightly  more technical.
We start with defining the multiplication operator.   For any function $W:X\rightarrow \mathbb{R}$,  we define $M_{W}$  by
$$(M_{W}f)(x)=W(x)f(x).$$
In what follows, we shall identify  the operator $M_W$ with  the function $W$. This means that, if $T$ is a linear operator,
we shall denote by  $W_1T$, $TW_2$, $W_1TW_2$,  the operators  $M_{W_1}T, TM_{W_2}$, $M_{W_1}TM_{W_2}$, respectively.

We can now formulate the  weighted $L^p-L^2$ estimates (Sobolev type conditions). Firstly we  consider
 \begin{equation*}\label{EVp}
 \tag{${\rm EV_{p,2}  }$}
\sup_{t>0}  \|   e^{-t^2L}\, {V_{{t}}^{1/p-1/2}} \|_{p \to 2} < +\infty,
\end{equation*}
where $V_t(x)=V(x,t)$ and $1 \le p< 2$. An detailed and systematic discussion on the condition \eqref{EVp} can founded
in \cite{BCS}. The following condition which was introduced  in \cite{COSY}:
\begin{equation*}\label{Gp}
 \tag{${\rm G_{p,2}  }$}
\big\|e^{-t^2L}\bchi_{B(x, s)}\big\|_{p\to 2} \leq
CV(x, s)^{{1\over 2}-{1\over p}} \left({s\over  {t}}\right)^{n({1\over p}-{1\over 2})}
\end{equation*}
holds for all $x\in X$ and  $s\geq t>0$.

\begin{lemma}\label {le2.0} Let $1\leq p< 2$.
 Suppose that $L$  satisfies
the property  \eqref{FS}.
Then the following are equivalent:

(i) \eqref{EVp} holds.

(ii) \eqref{Gp} holds.

(iii) For every  $N>n(1/p-1/2)$ there exists $C$ such that
\begin{eqnarray*}
 \big\|   (I+t \sqrt{L}\,)^{-N} V_t^{{\frac1p}-{\frac12}} \big\|_{p \to 2}
 \leq  C.
\end{eqnarray*}

(iv)
For all $x\in X$ and    $r\geq  t >0$   we have
\begin{equation*}
\big\|(I+t\SL\,\,)^{-N}\bchi_{B(x, r)}\big\|_{p\to 2} \leq
CV(x, r)^{{\frac 12}-{\frac 1p}} \left({\frac rt}\right)^{n\left({\frac 1p}-{\frac 12}\right)}.
\end{equation*}

\end{lemma}

\begin{proof}
The equivalence of the conditions $(ii)$ and $(iv)$ was verified in \cite[Proposition I.3]{COSY}.
The similar argument shows that the conditions $(i)$ and $(iii)$ are also
equivalent. Thus it is enough to show equivalence between  $(iii)$ and $(iv)$.

First we prove that
$(iii)$ implies $(iv)$. Note that by the doubling condition for all $y \in B(x,r)$ one has $V(x,r)\sim V(y,r)$.
 Hence for all $x\in X$ and    $r\geq  t >0$,
\begin{align*}
\big\|(I+t\SL\,\,)^{-N}\bchi_{B(x, r)}\big\|_{p\to 2}
&\le   C\big\|(I+t\SL\,\,)^{-N}\bchi_{B(x, r)}
  V_r^{{\frac 1p}-{\frac 12}}
  V(x, r)^{{\frac 12}-{\frac 1p}}
  \big\|_{p\to 2}
\\
&\le  C \big\|(I+t\SL\,\,)^{-N}\bchi_{B(x, r)}
  V_t^{{\frac 1p}-{\frac 12}}
  \big\|_{p\to 2} V(x, r)^{{\frac 12}-{\frac 1p}}
  \left({\frac rt}\right)^{n\left({\frac 1p}-{\frac 12}\right)}
  \\
  & \le  C \big\|(I+t\SL\,\,)^{-N}
  V_t^{{\frac 1p}-{\frac 12}}
  \big\|_{p\to 2} V(x, r)^{{\frac 12}-{\frac 1p}}
  \left({\frac rt}\right)^{n\left({\frac 1p}-{\frac 12}\right)} .
 \end{align*}
 By the assumption $(iii)$ it follows that
 \begin{align*}
  \big\|(I+t\SL\,\,)^{-N}\bchi_{B(x, r)}\big\|_{p\to 2} \le C V(x, r)^{{\frac 12}-{\frac 1p}}
  \left({\frac rt}\right)^{n\left({\frac 1p}-{\frac 12}\right)},
\end{align*}
where  we used  (iii) in the last inequality.

We now show that $(iv)$ implies $(iii)$. Let us recall the well known identity,  for $a>0$,
$$
C_a\int_0^\infty\left(1-\frac{x^2}{s}\right)^a_+e^{-s/4}s^a \,ds=e^{-x^2/4}
$$
with some suitable $C_a>0$.
Taking the Fourier transform on both sides of the above equality yields
$$
\int_0^\infty F_a(\sqrt s\lambda)  s^{a+\frac{1}{2}}e^{-s/4}ds=e^{-\lambda^2},
$$
where $F_a$ is the Fourier transform of the function $t \to (1-{t^2})^a_+$ multiplied
by the appropriate constant. Hence, by spectral theory,
$$
\int_0^\infty F_a(\sqrt{stL})  s^{a+\frac{1}{2}}e^{-s/4}ds=e^{-tL}.
$$
Using this and Minkowski's inequality give
\begin{eqnarray*}
	&&\quad \|   e^{-tL}\,  {V_{{\sqrt{t}}}^{{\frac 1p}-{\frac 12}}} \|_{p \to 2}
	\le \int_0^\infty \|   F_a(\sqrt{tsL})    {V_{{\sqrt{t}}}^{{\frac 1p}-{\frac 12}}} \|_{p \to 2} s^{a+\frac{1}{2}}e^{-s/4}ds
	\\
	 && \le C\int_0^\infty \|   F_a(\sqrt{tsL})    {V_{{\sqrt{st}}}^{{\frac 1p}-{\frac 12}}} \|_{p \to 2}
	\Big(\sqrt{s}+\frac{1}{\sqrt{s}}\Big)^{{\frac 1p}-{\frac 12}}s^{a+\frac{1}{2}}e^{-s/4}ds,
\end{eqnarray*}
hence, with $a$ large enough,
\begin{equation}\label{eqq}
\sup_{t>0} \|   e^{-tL}\,  {V_{{\sqrt{t}}}^{{\frac 1p}-{\frac 12}}} \|_{p \to 2}
\le C'\sup_{t>0}\|   F_a(\sqrt{tL})
{V_{{\sqrt{t}}}^{{\frac 1p}-{\frac 12}}} \|_{p \to 2}.
\end{equation}
We note that  $\Phi=F_a$ satisfies the assumptions of Lemma \ref{le2.2}. Thus
$
\mbox{supp} \, {F_a(r\sqrt {L})} \subseteq D_r,\ \forall
\,r>0.
$ 
Hence, by \cite[Lemma 4.1.2]{BCS}
\begin{equation}
\label{eqq1}
\|F_a(r\sqrt {L}){V_{{r}}^{{\frac 1p}-{\frac 12}}}\|_{p \to 2}
\le C \sup_{x\in M} \| F_a(r\sqrt {L}){V_{{r}}^{{\frac 1p}-{\frac 12}}}\bchi_{B(x,  r) }\|_{p\to 2}.
\end{equation}
Observe that
\begin{eqnarray*}
\| F_a(r\sqrt {L}){V_{r}^{{\frac 1p}-{\frac 12}}}\bchi_{B(x,  r) }\|_{p\to 2}
&\le&
\| F_a(r\sqrt {L})
(1+r\sqrt {L})^{N}
(1+r\sqrt {L})^{-N}
{V_{{r}}^{{\frac 1p}-{\frac 12}}}\bchi_{B(x,  r) }\|_{p\to 2}\\ &\le&
\| F_a(r\sqrt {L})
(1+r\sqrt {L})^{N}\|_{2\to 2}
\|(1+r\sqrt {L})^{-N}
{V_{{r}}^{{\frac 1p}-{\frac 12}}}\bchi_{B(x,  r) }\|_{p\to 2}\\
&\le& C \|(1+r\sqrt {L})^{-N}
{V_{{r}}^{{\frac 1p}-{\frac 12}}}\bchi_{B(x,  r) }\|_{p\to 2}.
\end{eqnarray*}
From this and  $(iv)$ with $r=t$, we get
\begin{eqnarray*}
\| F_a(r\sqrt {L}){V_{r}^{{\frac 1p}-{\frac 12}}}\bchi_{B(x,  r) }\|_{p\to 2}
 &\le&  C
{V(x,r)^{{\frac 1p}-{\frac 12}}}
\|(1+r\sqrt {L})^{-N}
\bchi_{B(x,  r) }\|_{p\to 2} \le C.
\end{eqnarray*}
Combining this with \eqref{eqq} and \eqref{eqq1} shows \eqref{EVp} which is equivalent with $(iii)$.
 \end{proof}

Recall  that $L$ is a non-negative self-adjoint operator on $L^2(X)$
and that
the semigroup $e^{-tL}$, generated by $-L$ on $L^2(X)$,  has the kernel  $p_t(x,y)$
which  satisfies
the following  Gaussian upper bound:
\begin{equation*}\label{ge}
\tag{GE}
\big|p_t(x,y)\big|\leq {C\over V(x,\sqrt{t})} \exp\left(-c {  d^2(x,y)\over    t } \right)
\end{equation*}
for all $t>0$,  and $x,y\in X,$   where $C$ and $ c$   are positive constants.
The stimate \eqref{ge} follows from \eqref{FS} and (${\color{red} \rm EV_{1,2}}$).
Indeed,  (${\color{red} \rm EV_{1,2}}$) is equivalent to the standard Gaussian
heat kernel estimate which is  valid for a broad class of second order
elliptic operators, see e.g. \cite{BCS}.

It is not difficult to see that, for $1\leq p<2$, both  the conditions \eqref{FS} and \eqref{EVp}
 follow from the Gaussian estimate \eqref{ge}. But the converse is not true in general. For some $1<p<2$, there
are operators which fail to  satisfy  \eqref{ge} while  \eqref{FS} and \eqref{EVp} hold for them. Examples for such
operators are provided by  the Schr\"odinger operators
with inverse-square potential, see \cite{CouS} and the second order elliptic operators with rough
lower order terms, see \cite{LSV}.

\noindent{\bf 2.2. Stein-Tomas restriction type condition.}
Let $1\leq p< 2$ and $2\leq
q\leq\infty$.  Following \cite{COSY}, we say that $L$ satisfies the Stein-Tomas
restriction type condition if  for any $R>0$ and all Borel functions $F$ supported   in $ [0,R],$
\begin{equation*}\label{st}
\tag{${\rm ST^q_{p, 2}}$}
\big\|F(\!\SL\,\,)\bchi_{B(x, r)} \big\|_{p\to 2} \leq CV(x,
r)^{{1\over 2}-{1\over p}} \big( Rr \big)^{n({1\over p}-{1\over
		2})}\big\| \delta_R F  \big\|_{q}
\end{equation*}
for all $x\in X$ and all $r\geq 1/R$.
To motivate this definition we state the following two lemmas.

\begin{lemma}\label{le2.3}
Assume that $C^{-1}r^n \leq V(x, r)\leq Cr^n$ for all $x\in X$ and $r>0$. Then ${\rm (ST^2_{p, 2})}$ is equivalent to ${\rm (R_p)}$.
\end{lemma}

\begin{lemma}\label{le2.4}
 Assume that a metric measure space $(X,d,\mu)$ satisfies the
	doubling condition \eqref{eq2.2}. Then ${\rm (ST^\infty_{p, 2})}$ is equivalent to \eqref{EVp} or any other condition  listed in Lemma~\ref{le2.0}.
\end{lemma}

For the proofs of these Lemmas  and more on the condition \eqref{st} we refer the reader to \cite{COSY}, especially
\cite[Proposition I.3]{COSY} and \cite[Proposition I.4]{COSY}.

The following  result for  the spectral multipliers of non-negative self-adjoint
operators was one of the main results obtained in  \cite[Theorem  I.16, Corollary I.6]{COSY}.
Fix a non-trivial auxiliary function $\eta \in C_c^\infty(0,\infty)$.

\begin{prop}  \label{prop2.4}
  Assume that   $L$ satisfies  the property {\rm (FS)}
and the condition  ${\rm (ST^{q}_{p, 2})}$ for some $p,q$ satisfying
$1\leq p<2$ and $2\leq q\leq \infty$. 

\begin{itemize}
\item[(i)]  Then for any   bounded Borel
function $F$ such that
$\sup_{t>0}\|\eta\, \delta_tF\|_{W^{\beta, q}}<\infty $ for some
$\beta>\max\{n(1/p-1/2),1/q\}$ the operator
$F(\!\SL\,\,)$ is bounded on $L^r(X)$ for all $p<r<p'$.
In addition,
\begin{eqnarray*}
   \|F(\!\SL\,\,)  \|_{r\to r}\leq    C_\beta\Big(\sup_{t>0}\|\eta\, \delta_tF\|_{W^{\beta, q}}
   + |F(0)|\Big).
\end{eqnarray*}

\item[(ii)] For all  ${\alpha}>  n(1/p-1/2)-1/q$  we have the uniform bound, for $R>0$,
\begin{eqnarray*} 
\Big\|\Big(I-{L\over R^2}\Big)_+^{\alpha}\Big\|_{p\to p}\leq C.
\end{eqnarray*}
\end{itemize}
\end{prop}

Finally, we state a standard  weighted inequality for the  Littlewood-Paley square function, which we shall use in what follows. For its proof, we refer the reader to
\cite{CD, DSY} for $p=1$, and  \cite{AM} for the general  $1\leq p<2$ on the Euclidean space  $\mathbb R^n$.
The estimate remains valid on
 spaces of homogeneous type.

\begin{prop}\label{prop2.5}
  Assume that   $L$ satisfies the  property {\rm (FS)}
and the condition  \eqref{EVp} for some
$1\leq p<2$.
 Let $\psi $  be  a function in $
{\mathscr S} ({\mathbb{R}})$ such that $\psi(0)=0$, and let the
quadratic functional be defined by
\begin{eqnarray*}
{\mathcal G}_L(f)(x)=\Big( \sum_{j\in{\mathbb Z}}  |\psi(2^j\SL\,\,)f(x)  |^2\Big)^{1/2}
\end{eqnarray*}
  for $f\in L^2(X)$.
Then for any $w\in A_1$ (i.e., the Muckenhoup $A_1$ weight), ${\mathcal G}_L$ is bounded on $L^{r}(w, X)$ for all $p<r<p'.$
\end{prop}

\section{ Plancherel estimate and  maximal Bochner-Riesz operator }
\setcounter{equation}{0}

In this section we will discuss the case $p=1$ for the condition \eqref{st}.  In Corollary  \ref{cor3.3} and
Proposition \ref{prop3.2} below we state a version of Theorem A which deals with the case $p=1$.
 In this case the proofs of results are significantly simpler.
We also describe  some other  observations
which will be useful for results in full generality. Following \cite{DOS}, we will call  the estimate  ${\red{ \color{red} \rm{ (ST^q_{1, 2})}}}$
the Plancherel estimate.

Assume that $(X, d, \mu)$ satisfies the doubling condition \eqref{eq2.2}.
We  start  with the following    lemma.

\begin{lemma}\label {le3.0}
Let    $L$  satisfy the Gaussian bound ${\rm (GE)}$ and let  $m$ be a bounded Borel function such
that $\supp m\subseteq [-2, 2]$.  If   $\|m\|_{W_{s}^{2}}<\infty$ for some $s>n +1/2$,   for all $x\in X,$
$$
\sup_{t>0}\left| m (tL)f(x)\right| \leq C\|m\|_{W_{ s}^{2}}  {\mathfrak M}(f)(x).
$$

As a consequence, if  $\alpha >n$,  then   $S_{\ast}^{ \alpha}(L)$ is a bounded operator on  $L^p(X)$ for all $ 1<p<\infty.$
 \end{lemma}

\begin{proof}
Let  $H(t):=m(\sqrt{t}) e^{t}$. By  the Fourier inversion formula
$
H (t) = \frac{1}{2\pi}\int_{-\infty}^{+\infty}\widehat{H }(\tau)
e^{it\tau}d\tau,
$
we have
\begin{eqnarray}\label{e3.02}
 m(t\sqrt{L}\,) =H(t^2L)e^{-t^2L}
=  \frac{1}{2\pi}\int_{-\infty}^{+\infty}\hat{H}(\tau)e^{- t^2(1-i\tau)L}d\tau.
\end{eqnarray}
Let  $z:=t^2(1-i\tau)$ and $\theta={\rm arg } z$.
From ${\rm (GE)}$, it is well known   (see \cite[Theorem 7.2]{Ou}) that
  there exist  positive constants $C, c$ such that for all $z\in{\mathbb C}^+$ and a.e. $x,y\in X,$
\begin{eqnarray}\label{e3.03}
\big|p_z(x,y)\big| &\leq& {C\left(\cos\theta\right)^{-n}\over \sqrt{ V\left(x, \sqrt{|z|\over \cos\theta} \right)
V\left(y, \sqrt{|z|\over \cos\theta} \right) }} \exp\left(-c {  d^2(x,y)\over   |z| }  \cos\theta\right).
\end{eqnarray}
By the  doubling  properties  of the space $X$, we use a standard argument to obtain
\begin{eqnarray*}
\big|e^{-zL} f(x)\big| &\leq& C (1+\tau^{2})^{n/2}\int_X    {1\over V\left(x, \sqrt{|z|\over \cos\theta} \right)}
\exp\Big(-{  {d(x,y)^2\over  2c\, |z| } \cos\theta}\Big)  |f(y)|d\mu(y)\nonumber\\
&\leq&
C(1+\tau^{2})^{n/2}{  \mathfrak M} f(x).
\end{eqnarray*}
Then, from this and \eqref{e3.02} it follows that,  for any  $\varepsilon\in(0,1)$,
\begin{eqnarray}\label{e3.04}
 |m(t\sqrt{L}\,)f(x)|
&=&\bigg|\frac{1}{2\pi}\int_{\mathbb{R}}e^{- zL}f(x)\hat{H}(\tau)d\tau\bigg|\nonumber
\le C {\mathfrak M}(f)(x)\int_{\mathbb{R}}|\hat{H}(\tau)|(1+\tau^{2})^{n/2}d\tau  \nonumber\\
&\leq&C\|H\|_{W_{ (2n+1)/2+\varepsilon}^{2} } {\mathfrak M}(f)(x)
\le C\|m\|_{W_{ (2n+1)/2+\varepsilon}^{2} }  {\mathfrak M}(f)(x).
\end{eqnarray}
\noindent
Because $\|H\|_{W_{ (2n+1)/2+\varepsilon}^{2} }
\leq C\|m\|_{W_{ (2n+1)/2+\varepsilon}^{2} }$ since $ \supp m\subset[-2,2]$.  This gives the desired inequality.

Finally, we notice that     $(1-t^2)_+^{\alpha}\in W^2_{s} $ if
and only if $  \alpha>s- 1/2$.
From (\ref{e3.04}), $L^p$-boundedness of  the Hardy-Littlewood maximal operator $\mathfrak M$,
  we see that for $\alpha >n$,    $S_{\ast}^{ \alpha}(L)$ is a bounded operator on  $L^p(X)$ for $ 1<p<\infty.$
\end{proof}

\newcommand{\mel}{{\mathlarger{ \mathfrak m}}}

In Lemma~\ref{le3.0} the order of $\alpha$ for which $S_{\ast}^{\alpha}(L)$ is bounded on  $L^p(X)$ for $ 1<p<\infty$
 is relatively large. This is mainly  because the maximal bound is obtained by the pointwise estimate. The bound can
 be improved by  making use of the spectral theory.  For this, let us first  recall that the Mellin transform of the
 function $F \colon \,   {\mathbb R}\to  {\mathbb C}$ is defined by
\begin{eqnarray}\label{e3.2}
\mel_F(u)=\frac{1}{2\pi}\int_0^\infty
F(\lambda)\lambda^{-1-iu}d\lambda,\quad u\in{\mathbb R}.
\end{eqnarray}
Moreover the inverse transform is given by the following formula
\begin{eqnarray}\label{e3.3}
F(\lambda)=\int_{\mathbb R} \mel_F(u)\lambda^{iu}du,\quad \lambda\in[0,\infty).
\end{eqnarray}

\begin{lemma}\label {le3.1}
Suppose that   $L$ satisfies  the
  property ${\rm (FS)}$ and the condition \eqref{EVp} for some
  $1\leq p< 2$.
Let $p<r<p'$ and $s>n|{1/p}-{1/2}|.$
Suppose also that $F \colon \,   {\mathbb R}\to  {\mathbb C}$ is a
bounded Borel
function    such that
$$
   \int_{\mathbb
R}|\mel_F(u)|(1+|u|)^{s}du  =C_{F,s} < \infty.
$$
Then  the maximal operator
\begin{eqnarray}\label{e3.4}
F^*(L)f(x)=\sup_{t>0}|F(tL)f(x)|
\end{eqnarray}
  is a
bounded operator  on $L^r(X)$  with
$
\|F^*(L)\|_{r \to r}
\le C C_{F,s}.
$
In particular, if $\supp F\subseteq [-2, 2]$ and   $\|F\|_{W_{s}^{2}}<\infty$ for some
$s>n|{1/ p}-{1/2}| +1/2$,  then $F^*(L)$ is bounded on  $L^r(X)$ for $ p<r<p'$. As a consequence, if
$\alpha >n|{1/ p}-{1/2}|$,  then     $S_{\ast}^{\alpha}(L)$ is bounded
on  $L^r(X)$ for $p<r<p'$.
 \end{lemma}

\begin{proof}
By \eqref{e3.2}, \eqref{e3.3} it follows
 that
\begin{eqnarray*}
F(tL)&=&\int_0^{\infty}F(t\lambda) d E_L(\lambda)
=\int_0^{\infty} \int_{\mathbb R} \mel_F(u)(t\lambda)^{iu}du \,  d E_L(\lambda)\\
&=&\int_{\mathbb R}\int_0^{\infty} \mel_F(u)(t\lambda)^{iu}  d E_L(\lambda) du=
\int_{\mathbb R} \mel_F(u) t^{iu}L^{iu}du.
\end{eqnarray*}
Hence
$
F^*(L)f(x)=\sup_{t>0}|F(tL)f(x)| \le  \int_{\mathbb R} |\mel_F(u)|   |L^{iu}f(x)|  du.
$
And we get
\begin{eqnarray}\label{e3.5}
\|F^*(L)f\|_{r}\leq  C \|f\|_r \int_{\mathbb R} |\mel_F(u)|   \|L^{iu}\|_{r\to r}  du.
\end{eqnarray}

From  Proposition~\ref{prop2.4}, we have that
  the imaginary power  $L^{iu}$
of $L$ is bounded  on $L^r(X), p<r<p'$  with the bound
$
\|L^{iu}\|_{ r \to  r} \le C (1+|u|)^{s}
$
for any $s>n|{1/p}-{1/2}|$. This, together  with \eqref{e3.5}, gives
\begin{eqnarray}\label{e3.7}
\|F^*(L)f\|_{r}\leq  C \|f\|_r \int_{\mathbb R} |\mel_F(u)|   (1+|u|)^{s} du.
\end{eqnarray}

Set $F(t)=(1-t^2)_+^{\alpha}.$
 Substituting  $\lambda=e^\nu$ in \eqref{e3.2}, we notice  that $m$ is the Fourier transform of $G(\nu)=F(e^{\nu})$.
Since  $(1-t^2)_+^{\alpha}$ is compactly supported in $[-1, 1]$, we get
\begin{eqnarray}\label{e3.8}
\int_{\mathbb R} |\mel_F(u)|   (1+|u|)^{s} du\leq
C \|G\|_{W^2_{s+1/2+\varepsilon} }\leq C\|F\|_{W^2_{s+1/2+\varepsilon} }
\end{eqnarray}
for any $\varepsilon>0$.
On the other hand,  $(1-t^2)_+^{\alpha}\in W^2_{s+1/2+\varepsilon} $ if
and only if $ \alpha>s+\varepsilon$.   From this, we know that
if  $\alpha>n|{1/p}-{1/2}|$,    then  $S_{\ast}^{\alpha}(L)$ is a bounded operator on  $L^r(X)$ for $ p<r<p'.$
\end{proof}

As a consequence of Lemma~\ref{le3.1},  we have the following which gives essentially sharp $L^2$ maximal bound for the Bochner-Riesz means.

\begin{cor}\label{cor22}
Suppose that   $L$ satisfies  the
  property ${\rm (FS)}$ and the condition  \eqref{EVp} for some
  $1\leq p< 2$.
  If  $\alpha >0$,  then     $S_{\ast}^{ \alpha}(L)$ is   bounded
 on  $L^2(X)$.
\end{cor}

In Lemma \ref{le3.0} we obtained the pointwise estimate for the maximal function under  the Gaussian bound ${\rm (GE)}$ only.
In what follows we additionally impose the condition ${\rm (ST^{q}_{1, 2})}$. This significantly improves the regularity
assumption on $F$ so that this allows us to essentially recover the sharp maximal bounds for the Bochner-Riesz means for
$p=\infty$,  see Corollary~\ref{cor3.3}. The proof of Proposition \ref{prop3.2}   was inspired  by an argument of Thangavelu \cite[Theorem 4.2]{Th1}.

\begin{prop}\label {prop3.2} Let  $q\in [2, \infty]$.
Let  $L$ satisfy the Gaussian bound ${\rm (GE)}$ and let $F$ be a Borel function such that $\supp F\subseteq [1/4, 1]$ and
  $\|F\|_{W_{s}^{q} }<\infty$ for some $s>n/2$.
Suppose that  the  condition  ${\rm (ST^{q}_{1, 2})}$ holds,  then we have, for each $2\leq r<\infty$,
\begin{equation*}
 F^{\ast} (L)f(x)\leq C   \|F\|_{W^q_{s} }  \mathfrak M_r f(x).
\end{equation*}
Hence, $F^*(L)$ is a bounded operator on  $L^p(X)$ for all $ 2<p<\infty$.
\end{prop}

\begin{proof}
Let $r'\in (1, 2]$ such that $1/r+ 1/r'=1$ and fix $R>0$.
 Consider a partition of $X$ into the dyadic annuli
$A_k=\{ y:  2^{k-1}R^{-1}< d(x,y)\leq 2^kR^{-1} \}$, for $k\in\mathbb N$.
For a given $f$ we set
\[
f_0(y)=f(y)\chi_{\{y: \, d(x,y)\leq R^{-1}\}}, \ \ f_k(y)=f(y)\chi_{A_k}(y), \ \ k\in\mathbb N.
\] Then, note that  $ \left|F\big({L/R^2}\big) f(x)\right|\leq   \sum_{k=0}^{\infty} |F(L/R^2) f_k(x)| $. By H\"older's inequality,
 for  $f\in L^2(X)\cap L^p(X)$,
  \begin{eqnarray*}
 \left|F\big({L/R^2}\big) f(x)\right|
  \leq
   \left(\int_{  d(x,y)\leq R^{-1}} |K_{F(L/R^2)} (x,y)|^{r'} d\mu(y)\right)^{1/r'} \|f_0\|_r+ \sum_{k=1}^{\infty} 
 \left(\int_{A_k}
|K_{F(L/R^2)} (x,y)|^{r'} d\mu(y)\right)^{1/r'}  \|f_k\|_r.
 \end{eqnarray*}
 We also note that
   \begin{eqnarray*}
 \left(\int_{A_k}
|K_{F(L/R^2)} (x,y)|^{r'} d\mu(y)\right)^{1/r'}
&\leq&
\sum_{k=0}^{\infty} V(x, 2^{k} R^{-1})^{{1\over r'}-{1\over 2}}\left(\int_{A_k} |K_{F(L/R^2)} (x,y)|^2
 d\mu(y)\right)^{1/2}
 \end{eqnarray*}
and $\|f_k\|_r\leq C V(x, 2^{k} R^{-1})^{1/r} \mathfrak M_r f(x)$.
Combining all these inequalities gives
  \begin{eqnarray}  \label{e3.16}
  \left|F\big({L/R^2}\big) f(x)\right|\le C\mathfrak M_r f(x) \sum_{k=0}^{\infty} 2^{-ks}V(x, 2^{k} R^{-1})^{1/2}\, \mathfrak I (R),
 \end{eqnarray}
where
\[  \mathfrak I (R)= \left(\int_{X} |K_{F(L/R^2)} (x,y)|^2 (1+Rd(x,y))^{2s}d\mu(y)\right)^{1/2}.\]

  Notice  that $L$ satisfies  the condition  ${\rm (ST^{q}_{1, 2})}$   for some $q\in [2, \infty]$.
 By \cite[Lemma 4.3]{DOS} we have
  \begin{eqnarray*} 
\int_X \big|K_{F(L/R^2)} (x,y)\big|^2 \big(1+Rd(x,y)\big)^{2s} d\mu(y)&\leq&
 C  V(x, R^{-1})^{-1}
 \|F\|^2_{W^q_{{s } +\varepsilon}}, \ \ \ \forall \varepsilon>0.
\end{eqnarray*}
Hence,  this and \eqref{e3.16}  yield
 \begin{eqnarray*}
|F(L/R^2)f(x)|
&\leq&
C\|F\|_{W^q_{{s } +\varepsilon}} \sum_{k=0}^{\infty}  2^{-ks} \left({V(x, 2^{k} R^{-1})\over  V(x, R^{-1})}\right)^{1/2}   \mathfrak M_r f(x).
\end{eqnarray*}
Since $s> n/2$,  we get
\begin{eqnarray*}
|F(L/R^2)f(x)| \leq  C\|F\|_{W^q_{{s } +\varepsilon}} \sum_{k=0}^{\infty}  2^{({n\over 2}-s)k}   \mathfrak M_r f(x)\leq
  C \|F\|_{W^q_{{s } +\varepsilon}}   \mathfrak M_r f(x)\,.
 \end{eqnarray*}
From this and $L^p$-boundedness of  the Hardy-Littlewood maximal operator ${\mathfrak  M}_r$ for $p>r$,
  we obtain that
 $  \| F^{\ast}(L) \|_{p\to p}\leq C$ for $p\in (2,\infty)$.
This completes the proof.
\end{proof}

We conclude this section with  the following result which covers a special case of Theorem A. In fact,  this
shows the case $p_0=1$ in  Theorem A if we take $q=2$ in the following.

\begin{cor}\label{cor3.3}
Let $L$ satisfy the Gaussian bounds ${\rm (GE)}$.
Suppose that  the  condition  ${\rm (ST^{q}_{1, 2})}$ holds for some $q\in [2, \infty]$.
If  $\alpha >n/2-1/q$,
then for each $2\leq r<\infty$,
\begin{equation*}
 S_{\ast}^{ \alpha }(L)f(x)\leq C  \mathfrak M_r f(x).
\end{equation*}
As a consequence, $S_{\ast}^{\alpha }(L)$ is a bounded operator on $L^p(X)$ for all $2\leq p\leq\infty$.
\end{cor}

\begin{proof}
  Let $ S^{\alpha }(t) = (1-t^2)^{\alpha }_+$.  We set
$$
 S^{\alpha }(t) =S^{\alpha }(t)  \phi(t^2) + S^{\alpha }(t)(1-\phi(t^2))=:S^{\alpha, 1 }(t^2) +S^{\alpha, 2 }(t^2),
$$
 where   $\phi\in C^{\infty}(\mathbb R)$ is supported in  $ \{ \xi:  |\xi| \geq 1/4 \}$
 and $\phi =1$ for all $|\xi|\geq 1/2$. Define the maximal Bochner-Riesz operators $S^{\alpha, i }_{ \ast} (L), i=1, 2$ by
 $$
 S^{\alpha, i }_{ \ast} (L) f(x) =
  \sup_{R>0} | S^{\alpha, i }  \big({L/R^2}\big)f(x)|, \ \  \ i=1,2.
  $$
 Note that by Lemma~\ref{le3.0},
$
 S^{\alpha, 2 }_{\ast} (L) f(x)  \leq C  \mathfrak M f(x).
$
For the operator $S^{\alpha, 1 }_{\ast} (L)$, we choose $n/2<s< \alpha +1/q$, and notice  that    $S^{ \alpha, 1 }\in W^q_{s+\varepsilon}$ if
and only if $ s+\varepsilon<\lambda +1/q$. Taking $\varepsilon$ small enough, we apply Proposition~\ref{prop3.2}
to obtain
that $
 S^{\alpha, 1}_{\ast} (L)f(x)\leq C   \|S^{ \alpha, 1 } \|_{W^q_{{s } +\varepsilon}} \mathfrak M_r f(x) \leq C \mathfrak M_r f(x)$
 for all $2\leq r<\infty.
$
Hence $S^{\alpha }_{\ast} (L)$ is bounded on $L^p(X)$ for all $p>2$.
This, together with Corollary~\ref{cor22}, finishes the proof of Corollary~\ref{cor3.3}.
 \end{proof}


\section{Spectral  restriction   estimate and maximal bound}\label{sec4}
 \setcounter{equation}{0}

The aim of this section is to prove  Theorem A. However, we would like to describe a slightly more general
result which remains valid for the spaces of homogeneous type. For this end,  we assume  that $(X, d, \mu)$
satisfies the doubling condition, that is \eqref{eq2.2}.
In this section, we will prove the following result, which yields  Theorem A as
a special case with $q=2$ and the uniform volume estimate  \eqref{eq1.1}.

 \begin{thm}\label{th1.1}
 Suppose  that $(X, d, \mu)$ satisfies the doubling condition \eqref{eq2.2}.
 	Suppose    that   $L$  satisfies the   property ${\rm (FS)}$
 	and the condition  ${\rm (ST^{q}_{ p_0, 2})}$ for some   $1\leq p_0 <2$ and $2\leq q\leq \infty$.
 	Then the operator $S_{\ast}^{\alpha}(L)$ is bounded on $L^p(X)$ whenever
 	\begin{eqnarray}\label{e4.0}
 	2\leq p< p_0'  \ \ \ {\rm and }\ \ \  \alpha>
 	\max\left\{ n\left({1\over p_0}-{1\over 2}\right)- {1\over q}, \, 0 \right\}.
 	\end{eqnarray}
 	As a consequence, if $f\in L^p(X)$,   for  $p$ and $\alpha$ in the  range of  \eqref{e4.0},
 	$$
 	\lim\limits_{R\to \infty}S_{R}^{\alpha}(L)f(x)=f(x), \ \ \ a.e.
 	$$
 	 \end{thm}

In order to prove Theorem~\ref{th1.1} we use the classical approach which makes use of the square function to control
the maximal operator (see \cite{Ca, C1, Lee}).  Here we should mention that we may assume 
that  
 \begin{equation} \label{nontrivial}  n\left({1\over p_0}-{1\over 2}\right)- {1\over q}\ge  0. \end{equation}
 Otherwise, by \cite[Corollary I.7]{COSY}  it follows that $L=0.$ Thus Theorem \ref{th1.1} trivially holds. We assume the condition \eqref{nontrivial} for the rest of this section.


\noindent{\bf 4.1 Reduction to square function estimate.}
Let us recall  the well known identity, for $\alpha>0,$
\begin{eqnarray*}\label{e4.1}
	\left(1-{|m|^2\over R^2}\right)^{\alpha }=C_{\alpha, \, \rho} R^{-2\alpha }\int_{|m|}^R (R^2-t^2)^{\alpha-\rho-1}t^{2\rho+1}
	\left(1-{|m|^2\over t^2}\right)^{ \rho}dt
\end{eqnarray*}
with $C_{\alpha, \, \rho}=2\Gamma(\alpha+1)/\Gamma(\rho+1)\Gamma(\alpha-\rho)$. By the spectral theory, we use an argument in \cite[p.278--279]{SW}
to obtain
\begin{eqnarray}\label{e4.2}
	S_{\ast}^{\alpha}(L)f(x)\leq  C'_{\alpha,\, \rho} \sup_{0<R<\infty}  \left( {1\over R}\int_0^R |S_t^{\rho}(L) f(x)|^2dt\right)^{1/2}
\end{eqnarray}
 provided that $\rho>-1/2$ and $\alpha  >\rho+1/2$.

By dyadic decomposition, we   write $x^{\rho}_+=\sum_{k\in{\mathbb Z}} 2^{-k\rho} \phi(2^{k}x)$ for some $\phi\in C_0^{\infty}(1/4, 1/2)$. Thus
\begin{eqnarray}\label{e4.3}
(1-|\xi|^2)_+^{\rho}=:
\phi_0^{\rho} (\xi)+\sum_{k=1}^{\infty} 2^{-k\rho} \phi_k^{\rho} (\xi)
 \end{eqnarray}
 where $\phi_k^{\rho}=\phi(2^k(1-|\xi|^2)), k\geq 1$  and
 \begin{eqnarray*}
{\rm supp}\ \phi_0^{\rho}&\subseteq& \{ |\xi|\leq {3\over 4}\}, \nonumber\\[4pt]
{\rm supp}\ \phi_k^{\rho}&\subseteq& \{ 1-2^{-k}\leq |\xi|\leq 1-2^{-k-2}\}.
 \end{eqnarray*}
By \eqref{e4.2},   for $\alpha>\rho+1/2$
  \begin{eqnarray}\label{e4.5}
\|S_{\ast}^{\alpha}(L)f\|_p
&\leq& C \left\|
\left(\sup_{0<R<\infty} {1\over R}\int_0^R \Big|\phi_0^{\rho}
\left({\sqrt{L}\over t}\right) f(x)\Big|^2dt\right)^{1/2}\right\|_p
\\ &+& C\sum_{k=1}^{\infty} 2^{-k\rho}\left\|\left(\int_0^{\infty}
\Big|\phi_k^{\rho} \left({\sqrt{L}\over t}\right) f(x)\Big|^2{dt\over t}\right)^{1/2} \right\|_p.  \nonumber
 \end{eqnarray}
By Lemma~\ref{le3.1}, for the first term we have
  \begin{eqnarray}\label{e4.6}
 \left\|\sup_{0<t<\infty} \left|\phi_0^{\rho}
\left(t{\sqrt{L} }\right) f(x)\right|  \right\|_p \leq C\|f\|_p, \ \ \ \  p_0<p<p'_0.
 \end{eqnarray}

Now, in order to prove Theorem~\ref{th1.1}, by \eqref{e4.5} it is sufficient to show the following.

\begin{prop}\label{prop4.1} Let $\phi$ be a fixed $C^{\infty}$ function supported in $[-1/2, 1/2], |\phi|\leq 1$.  For every $0<\delta\leq 1,$
 define
 \begin{eqnarray}\label{e4.7}
 T_{\delta}f(x)=\left( \int_0^{\infty} \Big|\phi\left(\delta^{-1}\left(1-{ {L}\over t^2} \right) \right)f(x)\Big|^2
   {dt\over t}\right)^{1/2}.
   \end{eqnarray}
Suppose  that   $L$ satisfies   the property  {\rm (FS)}
and the condition  ${\rm (ST^{q}_{p_0, 2})}$ for some $1\leq p_0<2$ and $ 2\leq q\leq \infty.$ 
%
Then  for all $2\leq p<p'_0 $ and $0<\delta\leq 1$,
 \begin{eqnarray}\label{e4.8}
   \|T_{\delta}f\|_{p}
  &\leq& C(p)\delta^{{1\over 2}+{1\over q}+n({1\over 2}- {1\over p_0})}   \|f\|_{p}.
  \end{eqnarray}
\end{prop}

Before we start the proof of Proposition~\ref{prop4.1},   we show that Theorem~\ref{th1.1} is a straightforward consequence of
Proposition~\ref{prop4.1}.
  
\begin{proof}[Proof of Theorem ~\ref{th1.1}]  Substituting \eqref{e4.6} and \eqref{e4.8} with $\delta=2^{-k}$ back into \eqref{e4.5} yields
that, for a small enough  $\varepsilon>0$,
\begin{eqnarray*}
	\|S_{\ast}^{\alpha}(L)f\|_p&\leq&  C\|f\|_p+
	C  \sum_{k=1}^{\infty} 2^{-k(\alpha-{1\over 2}-\varepsilon)}2^{-k({1\over 2}+{1\over q}+n({1\over 2}-{ 1\over p_0}))}
	\|f\|_p \leq C\|f\|_p
\end{eqnarray*}
provided that $2\leq p<p'_0$ and
$
\alpha>n\left({1/p_0}-{1/2}\right)-{1/q}.
$
This gives Theorem~\ref{th1.1}.
\end{proof}

In order to prove   Proposition~\ref{prop4.1}, let us   verify  \eqref{e4.8} for $p=2$ first. 
 Note that $\phi$ is a fixed $C^{\infty}$ function supported in $[-1/2, 1/2], |\phi|\leq 1$.
 It follows from the spectral theory \cite{Yo}
that, for any $f\in L^2(X)$,
\begin{eqnarray}\label{einter}
  \|T_\delta f\|_{2}
&=&\left\{\int_0^{\infty}\Big\langle\,\phi^2\left(\delta^{-1}\left(1-{ {L}\over t^2} \right)\right)f, f\Big\rangle {dt\over t}\right\}^{1/2}\nonumber 
 =\left\{\int_0^{\infty}\phi^2\left(\delta^{-1}\left(1-  t^2 \right)\right) {dt\over t}\right\}^{1/2}\|f\|_{2}\nonumber \\
 &\leq &
 C \delta^{\frac{1}{2}}\|f\|_{2}.
\end{eqnarray}
Since $\delta\in (0, 1]$ and we assume the condition \eqref{nontrivial}, 
 the estimate \eqref{e4.8} for $p=2$ follows from \eqref{einter}.

For proof of  Proposition~\ref{prop4.1} for $2<p<p_0'$, we make use of a weighted inequality which reduces the desired 
inequality to $L^2$ weighted estimate.  See \cite{Ca, C1, LRS1}.

\noindent{\bf 4.2. Weighted inequality for the square function.}
Let $r_0$ be a number such that $1/r_0 = 2/p_0 -1.$

\begin{lemma} \label{le4.1}  Suppose  that   $L$ satisfies   the property  {\rm (FS)}
and the condition  ${\rm (ST^{q}_{p_0, 2})}$ for some $1\leq p_0<2$ and $2\leq q\leq \infty$.
For any  $0\leq w$ and $0<\delta\le 1$,
\begin{eqnarray}\label{e4.10}
\int_X |T_\delta f(x)|^2 w(x) d\mu(x)  \leq C\delta^{1+{2\over q}+n({ 1- {{2}\over p_0} })} \int_X |f(x)|^2 \mathfrak M_{r_0}w(x)d\mu(x).
  \end{eqnarray}
\end{lemma}

Again before we prove the lemma we show that it concludes the proof of  Proposition~\ref{prop4.1}.
  For every  $2< p<p_0',$
we take  $w\in L^{r}$ with $\|w\|_r\leq 1$ where $1/r + 2/p=1$. Since   $r_0< r$,  we have
\begin{align*}
	&\int_X| T_{\delta}f(x)|^2 w(x)d\mu(x)
	\leq C\delta^{ 1+{2\over q}+n( 1-{2\over p_0}) } \int_X   |f(x)|^2 \mathfrak M_{r_0} w(x)d\mu(x)\nonumber\\
	&\leq C\delta^{ 1+{2\over q}+n( 1-{2\over p_0}) } \| f  \|_{p }^2\|\mathfrak M_{r_0} w\|_r\nonumber
	\leq C\delta^{ 1+{2\over q}+n( 1-{2\over p_0}) } \| f  \|_{p }^2.
\end{align*}
Hence
$$
\|T_{\delta}f\|^2_p   \leq  C\delta^{ 1+{2\over q}+n( 1-{2\over p_0}) } \| f  \|^2_{p }
$$
for some constant $C>0$ independent of $f$ and $\delta.$ This finishes the proof of Proposition~\ref{prop4.1}.

\begin{proof}[Proof of Lemma \ref{le4.1}]
 We start with  Littlewood-Paley decomposition associated to the operator $L$.
	Fix a  function $\varphi\in C^{\infty}  $  supported in $\{ 1\leq |s|\leq 3\}$ such that $\sum_{-\infty}^{\infty}\varphi(2^ks)=1$
	on ${\mathbb R}\backslash \{0\}$.  Let
	\begin{eqnarray}\label{e4.9}
	\varphi_k(\sqrt{L}\,)f= \varphi(2^{-k}\sqrt{L}\,) f,  \ \ \ \ k\in{\mathbb Z}.
	\end{eqnarray}
	By the spectral theory we have that, for any
	$f\in L^2(X),$
	\begin{eqnarray}\label{e44}
	\sum_k \varphi_k(\sqrt{L}\,)f=f\,.
	\end{eqnarray}
	%
		By \eqref{e44}  we have that, for $f\in L^2(X)\cap L^p(X)$,
 \begin{eqnarray} \label{e4.11}
 | T_{\delta}f(x)|^2
 &\leq&
  5\sum_k  \int_0^{\infty} \Big|\phi\left(\delta^{-1}\left(1-{ {L}\over t^2} \right) \right)\varphi_k(\sqrt{L}\,)f(x)\Big|^2
  {dt\over t} \\  &=&
  5\sum_k  \int_{2^{k-1}}^{2^{k+2}}  \Big|\phi\left(\delta^{-1}\left(1-{ {L}\over t^2} \right) \right)\varphi_k(\sqrt{L}\,)f(x)\Big|^2
  {dt\over t}. \nonumber
  \end{eqnarray}
 For  given  $0<\delta\leq 1$,  we set $j_0=-[\log_2\delta]-1$.
Fix an even function  $\eta\in C_0^{\infty}$, identically one on
$\{|s| \leq 1 \}$  and supported on $\{|s| \leq 2 \}$.    Let  us set
\begin{equation}
\label{zeta}
\zeta_{j_0}(s)=\eta(2^{-j_0} s),  \  \ \  \zeta_j(s)=\eta(2^{-j} s)-\eta(2^{-j+1} s),  \   j> j_0
\end{equation}
 so that
\begin{equation}
\label{id}
1\equiv \sum_{j\geq j_0 }\zeta_j(s), \ \ \ \ \forall s>0.
\end{equation}
Then we set $\phi_{\delta}(s)={ \phi\left(\delta^{-1}\left(1-|s|^2 \right) \right)}$, and set, for
 $ j\geq j_0$
   \begin{eqnarray}\label{e4.12}
  \phi_{\delta,j}(s)={1\over 2\pi}  \int_{-\infty}^{\infty}\zeta_j(u)
  {\widehat{ \phi_{\delta} }}(u) \cos(s u) du.
  \end{eqnarray}
  Note that  $\zeta_j$ is  a dilate of a fixed smooth compactly supported  function, supported away from $0$  when $j>j_0$, hence
    \begin{eqnarray}\label{e4.13}
  | \phi_{\delta,j}(s)|\leq
  \left\{
  \begin{array}{ll}
  C_N 2^{(j_0-j)N} ,  & |s|\in [1/4, 8];\\[8pt]
   C_N  2^{j-j_0}  (1+ 2^j |s-1|)^{-N}, & {\rm otherwise}
  \end{array}
  \right.
  \end{eqnarray}
  for any $N$ and all  $j\geq j_0$   (see    \cite[page 18]{C1}).
 By the Fourier inversion formula,
 \begin{eqnarray}\label{e4.14}
 \phi\left(\delta^{-1}\left(1-{s}^2 \right)\right)= \sum_{j\geq j_0}\phi_{\delta,j}(s), \ \ \ \ s>0.
  \end{eqnarray}

Set \[d_j=2^{j+1}/t.\] By Lemma~\ref{le2.2},
\begin{eqnarray}\label{e4.15}
  \supp K_{\phi_{\delta,j}(\sqrt{L}/t)}\subseteq \D_{d_j} =\left\{(x,y)\in X\times X:  \ d(x,y)\leq 2^{j+1}/t\right\}.
\end{eqnarray}
From \eqref{e4.11}, \eqref{e4.14} and Minkowski's inequality,  it follows  that for every function $w\geq 0,$
 \begin{eqnarray}\label{e4.16}
  \int_X | T_{\delta}f(x)|^2 w(x) d\mu(x)
  &\leq&
  C\sum_k\left[ \sum_{j\geq j_0}   \left(\int_{2^{k-1}}^{2^{k+2}}
  \left\langle  \Big|\phi_{\delta,j}\left({\sqrt{L}\over t} \right)
  \varphi_k(\sqrt{L}\,)f\Big|^2, \  w\right\rangle   {dt\over t} \right)^{1/2}\right]^{2}.
 \end{eqnarray}

For a given $k\in{\mathbb Z}, j\geq j_0$, set $\rho=2^{j-k+2}>0$.  Following   an argument as in \cite{GHS},  we can choose
 a sequence $(x_m)  \in X$ such that
$d(x_m,x_\ell)> \rho/10$ for $m\neq \ell$ and $\sup_{x\in X}\inf_m d(x,x_m)
\le \rho/10$. Such sequence exists because $X$ is separable.
Let $B_m=B(x_m, 3\rho)$ and define $\widetilde{B_m}$ by the formula
$$\widetilde{B_m}=\bar{B}\left(x_m,\frac{\rho}{10}\right)\setminus
\bigcup_{\ell<m}\bar{B}\left(x_\ell,\frac{\rho}{10}\right),$$
where $\bar{B}\left(x, \rho\right)=\{y\in X \colon d(x,y)
\le \rho\}$.
Note that for $m\neq \ell$,
 $B(x_m, \frac{\rho}{20}) \cap B(x_\ell, \frac{\rho}{20})=\emptyset$. Hence, by the doubling condition \eqref{eq2.2}  
\begin{equation}\label{kk}
K=\sup_m\#\{\ell:\;d(x_m,x_\ell)\le  2\rho\} \le
  \sup_x  {V(x, (2+\frac{1}{20})\rho)\over
  V(x, \frac{\rho}{20})}< C  (41)^n.
\end{equation}
It is not difficult to see that
$$
\D_{\rho}  \subset \bigcup_{\{\ell, m:\, d(x_\ell, x_m)<
 2 \rho\}} \widetilde{B}_\ell\times \widetilde{B}_m \subset \D_{4\rho}.
$$
Recall that  $1/r_0 +2/p'_0=1$ and $d_j=2^{j+1}/t$.
It follows by \eqref{e4.15} and  H\"older's inequality  that, for every $j,k$ and any test function $w\geq 0,$
     \begin{eqnarray*}
 \left\langle  \Big|\phi_{\delta,j}\left({\sqrt{L}\over t} \right)
   \varphi_k(\sqrt{L}\,)f\Big|^2, \  w\right\rangle &=&
  \left\langle \Big| \sum_{\ell, \, m: \,  d(x_\ell, x_m)<2 d_j} \bchi_{\widetilde{B}_\ell}
   \phi_{\delta,j}\left({\sqrt{L}\over t}\right)  \bchi_{\widetilde{B}_m} \varphi_k(\sqrt{L}\,)  f
 \Big|^2, \,    w   \right \rangle. \nonumber
 \end{eqnarray*}
 Using \eqref{kk}, we have
 \begin{eqnarray*}
 \left\langle  \Big|\phi_{\delta,j}\left({\sqrt{L}\over t} \right)
   \varphi_k(\sqrt{L}\,)f\Big|^2, \  w\right\rangle  &=&
   \sum_\ell\left\langle \Big| \sum_{m: \,  d(x_\ell, x_m)<2d_j}  \bchi_{\widetilde{B}_\ell}
    \phi_{\delta,j}\left({\sqrt{L}\over t}\right)   \bchi_{\widetilde{B}_m} \varphi_k(\sqrt{L}\,)  f
  \Big|^2, \,    w   \right \rangle \nonumber\\
  &\leq &  K   \sum_{\ell } \sum_{ m:\, d(x_\ell, x_m)<2 d_j}
   \left\langle
  \big|  \bchi_{\widetilde{B}_\ell}
   \phi_{\delta,j}\left({\sqrt{L}\over t}\right)   \bchi_{\widetilde{B}_m} \varphi_k(\sqrt{L}\,)  f
 \big|^2, \,    w   \right\rangle.
 \end{eqnarray*}
  By H\"older's inequality it follows that
  \begin{eqnarray}\label{e4.17}
  \qquad \left\langle  \Big|\phi_{\delta,j}\left({\sqrt{L}\over t} \right)
   \varphi_k(\sqrt{L}\,)f\Big|^2, \  w\right\rangle \leq K^2   \sum_{m }   \|\bchi_{B_m} w\|_{{r_0}}     \left\|\bchi_{B_m}
 \phi_{\delta,j}\left({\sqrt{L}\over t}\right)
 \bchi_{\widetilde{B}_m} \varphi_k(\sqrt{L}\,)  f \right\|^2_{{p'_0}}.
  \end{eqnarray}
Since $  \phi_{\delta,j}$ is not compactly supported, we choose an even
  function $\theta\in C_0(-4, 4) $
 such that $\theta(s)=1$  for $s\in (-2, 2)$. Set
 \begin{equation}\label{psi}
 { \psi}_{0, \delta}(s)=\theta(\delta^{-1}(1-s)) \ \ \ {\rm and}\ \
 { \psi}_{\ell, \delta}(s)=\theta(2^{-\ell}\delta^{-1}(1-s)) - \theta(2^{-\ell+1}\delta^{-1}(1-s))
\end{equation}
 for all $\ell\geq 1$ such that $1=\sum_{\ell=0}^{\infty} { \psi}_{\ell, \delta}(s)$, and so
 $
  \phi_{\delta,j}(s)=  \sum_{\ell=0}^{\infty}  \big ( { \psi}_{\ell, \delta}\phi_{\delta,j}\big)(s)
  $ for all $s>0.
 $
From this, we apply  \eqref{e4.17} to write
   \begin{eqnarray}\label{e4.18}\hspace{0.5cm}
 &&\hspace{-1.2cm} \left(\int_{2^{k-1}}^{2^{k+2}}
  \left\langle  \Big| \phi_{\delta,j}\left({\sqrt{L}\over t} \right)
  \varphi_k(\sqrt{L}\,)  f\Big|^2, \  w\right\rangle   {dt\over t} \right)^{1/2}\nonumber\\
 &\leq & \sum_{\ell=0}^{ [-{\rm log_2{\delta}}]}
  \left(   \sum_{m }   \|\bchi_{B_m}w\|_{{r_0}} \int_{2^{k-1}}^{2^{k+2}}  \left\|\bchi_{B_m}
 \left ( { \psi}_{\ell, \delta}\phi_{\delta,j}\right)\left({\sqrt{L}\over t}\right)
 \bchi_{\widetilde{B}_m} \varphi_k(\sqrt{L}\,)  f \right\|^2_{{p'_0}} {dt\over t}\right)^{1/2}
\nonumber\\
 &\quad\qquad +&
  \sum_{\ell= [-{\rm log_2{\delta}}]+1}^{\infty}
 \left(  \sum_{m }   \|\bchi_{B_m} w\|_{{r_0}}  \int_{2^{k-1}}^{2^{k+2}}  \left\|\bchi_{B_m}
 \left ( { \psi}_{\ell, \delta}\phi_{\delta,j}\right)\left({\sqrt{L}\over t}\right)
 \bchi_{\widetilde{B}_m} \varphi_k(\sqrt{L}\,)  f \right\|^2_{{p'_0}} {dt\over t}\right)^{1/2}\nonumber\\[6pt]
   &=&   I(j,k) + I\!I(j,k).
  \end{eqnarray}
As to be seen later,  the first term $I(j,k)$ is the major one.

  \noindent
  {\bf Estimate for $I(j,k)$.}
 For  $k\in{\mathbb Z}$ and $\lambda=0, 1, \ldots, \lambda_0=[{8/\delta}] +1$,
we set
\begin{equation}\label{ijk1}
I_\lambda=\left[2^{k-1} + \lambda 2^{k-1}\delta, \, 2^{k-1} + (\lambda+1)2^{k-1}\delta\right],
\end{equation}
   so that
  $[2^{k-1}, 2^{k+2}]\subseteq \cup_{\lambda=0}^{\lambda_0} I_\lambda$.
Define
  \begin{eqnarray}\label{ijk2}
\eta_\lambda(s) = \eta\left( \lambda+{ 2^{k-1} -s\over 2^{k-1}\delta}\right),
  \end{eqnarray}
where $\eta\in C_0^{\infty}(-1, 1)$ and $\sum_{\lambda\in {\mathbb Z}} \eta(\cdot-\lambda)=1$. Observe that  for every $t\in I_\lambda$,
it is possible that  $
{\psi}_{\ell, \delta} \left({s/t}\right)\eta_{\lambda'} (s)\not=0$ only when  $\lambda-2^{\ell+6}\leq \lambda'\leq \lambda+2^{\ell+6}.$
Hence, for $t\in I_\lambda$,
 \begin{eqnarray}
 \label{eqn4}
 \left ( { \psi}_{\ell, \delta}\phi_{\delta,j}\right)\left({\sqrt{L}\over t}\right)
 &=&   \sum_{\lambda'=\lambda-2^{\ell+6}}^{\lambda+2^{\ell+6}}
 \left ({ \psi}_{\ell, \delta}\phi_{\delta,j}\right)\left({\sqrt{L}\over t}\right)
 \eta_{\lambda'}  (\sqrt{L}\,),
  \end{eqnarray}
  so
 \begin{eqnarray*}
 I(j,k)\leq  \sum_{\ell=0}^{ [-{\rm log_2{\delta}}]} \left[ \sum_{m }  \|\bchi_{B_m} w\|_{{r_0}}  \sum_{\lambda}  \int_{I_\lambda}
 \left(\sum_{\lambda'=\lambda-2^{\ell+6}}^{\lambda+2^{\ell+6}}  \left\|\bchi_{B_m}
 \left ({ \psi}_{\ell, \delta}\phi_{\delta,j}\right)\left({\sqrt{L}\over t}\right)
 \eta_{\lambda'}  (\sqrt{L}\,)\big[\bchi_{\widetilde{B}_m} \varphi_k(\sqrt{L}\,)  f\big] \right\|_{{p'_0}}\right)^2  {dt\over t}\right]^{1/ 2}.
  \end{eqnarray*}
Note  that  $$ \supp { \psi}_{\ell, \delta} \subseteq (1-2^{\ell+2}\delta, 1+2^{\ell+2}\delta).$$
 Moreover,  if $\ell\geq 1$,   then  ${ \psi}_{\ell, \delta}(s)=0$  for $s\in (1-2^{\ell}\delta, 1+2^{\ell}\delta)$.
By the Stein-Tomas restriction type  condition ${\rm (ST^{q}_{p_0, 2})}$, we have, for $0\leq \ell\leq [-{\rm log_2{\delta}}]$,
 \begin{eqnarray} \label{e4.19} \hspace{1cm}
 && \left\|\bchi_{B_m}
 \left (\psi_{\ell, \delta}\phi_{\delta,j}\right)\left({\sqrt{L}\over t}\right)\right\|_{2\to {p'_0} }=
\left\| \left ({  \psi}_{\ell, \delta}\phi_{\delta,j}\right) \left({\sqrt{L}\over t}\right) \bchi_{B_m}   \right\|_{p_0\to 2}
\\
   &\leq& C
   \left(2^j(1+2^{\ell+2}\delta)\right)^{n ({1\over p_0}-{1\over 2})}\mu(B_m)^{{1\over 2}-{1\over p_0}}
   \|({  \psi}_{\ell, \delta}\phi_{\delta,j})\big((1+2^{\ell+2}\delta) \cdot\big)\|_q\nonumber
   \\
   &\leq& C
 2^{jn ({1\over p_0}-{1\over 2})}\mu(B_m)^{{1\over 2}-{1\over p_0}}
   \|({  \psi}_{\ell, \delta}\phi_{\delta,j})\big((1+2^{\ell+2}\delta) \cdot\big)\|_q. \nonumber
         \end{eqnarray}	
	
From the definition of the function ${ \psi}_{\ell, \delta}$, it follows by \eqref{e4.13}  that, for any $N<\infty$,
  \begin{equation}
  \label{psidel}
   \|{  \psi}_{\ell, \delta}\phi_{\delta,j}\big((1+2^{\ell+2}\delta) \cdot\big)\|_q \le  C
   \begin{cases}
     \delta^{1\over q} 2^{(j_0-j)N},  &\ell=0,
     \\[6pt]
    \delta^{1\over q}  2^{\ell\over q}  2^{j-j_0}\big(2^{j+\ell} \delta\big)^{-N-1},  &1\leq \ell\leq [-{\rm log_2{\delta}}],
     \\[6pt]
	  \delta^{1\over q} 2^{(j_0-j)N} 2^{- \ell N }, & 0\leq \ell\leq [-{\rm log_2{\delta}}].
	  \end{cases}
	\end{equation}
	By this we have
	\begin{eqnarray} \label{e4.20}
 &&\hspace{-1.8cm}   \left\|\bchi_{B_m}
 \left ({ \psi}_{\ell, \delta}\phi_{\delta,j}\right)\left({\sqrt{L}\over t}\right)
 \eta_{\lambda'}  (\sqrt{L}\,)\big[\bchi_{\widetilde{B}_m} \varphi_k(\sqrt{L}\,)  f\big] \right\|_{{p'_0}}\\
 &\leq&
 C\delta^{1\over q}  2^{(j_0-j)N}
   2^{jn ({1\over p_0}-{1\over 2})} 2^{- \ell N} \mu(B_m)^{{1\over 2}-{1\over p_0}}   \big\|
 \eta_{\lambda'}  (\sqrt{L}\,)\big[\bchi_{\widetilde{B}_m} \varphi_k(\sqrt{L}\,)  f\big] \big\|_{2}.\nonumber
  \end{eqnarray}
On the other hand,
 \begin{align*} 
 \sum_{\lambda}   &\int_{I_\lambda}  \left(\sum_{\lambda'=\lambda-2^{\ell+6}}^{\lambda+2^{\ell+6}}   \big\|
 \eta_{\lambda'}  (\sqrt{L}\,)\big[\bchi_{\widetilde{B}_m} \varphi_k(\sqrt{L}\,)  f\big] \big\|_{2}\right)^2  {dt\over t}
  \leq  C2^{\ell}\left({2^k\delta\over 2^{k}}\right) \sum_{\lambda}      \sum_{\lambda'=\lambda-2^{\ell+6}}^{\lambda+2^{\ell+6}}   \big\|
 \eta_{\lambda'}  (\sqrt{L}\,)\big[\bchi_{\widetilde{B}_m} \varphi_k(\sqrt{L}\,)  f\big] \big\|^2_{2}
 \nonumber
 \\
   &\qquad\qquad\qquad\leq  C2^{2\ell} \delta    \sum_{\lambda' }   \big\|
 \eta_{\lambda'}  (\sqrt{L}\,)\big[\bchi_{\widetilde{B}_m} \varphi_k(\sqrt{L}\,)  f\big] \big\|^2_{2}
 \leq C2^{2\ell} \delta  \big\|
 \bchi_{\widetilde{B}_m} \varphi_k(\sqrt{L}\,)  f  \big\|^2_{2}.
  \end{align*}
This, together with estimates \eqref{e4.19} and \eqref{e4.20}, the fact that $1/r_0 +2/p'_0=1$ and
 $$
\|\bchi_{B_m} w\|_{{r_0}} \leq C\mu(B_m)^{ {2\over p_0}-1} \inf_{x\in   B_m } \mathfrak M_{r_0} w(x),
$$
show  that
 \begin{eqnarray*}
 I(j,k)
 &\leq &  C\delta^{{1\over q}+{1\over 2}}  2^{(j_0-j)N}
   2^{jn ({1\over p_0}-{1\over 2})}
   \sum_{\ell}     2^{- \ell (N-1)}
	\left(\sum_{m }   \mu(B_m)^{{1 }-{2\over p_0}} \|\bchi_{B_m} w\|_{{r_0}} \int_{\widetilde{B}_m}  \big|
  \varphi_k(\sqrt{L}\,)  f |^2 d\mu(x)\right)^{1/2}  \nonumber\\
	&\leq &C\delta^{{1\over q}+{1\over 2}}  2^{(j_0-j)N}
   2^{jn ({1\over p_0}-{1\over 2})}
 \left(\sum_m
    \int_{{\widetilde{B}_m}}
 |\bchi_{\widetilde{B}_m} \varphi_k(\sqrt{L}\,)  f(x)|^2 \mathfrak M_{r_0} w(x)d\mu(x) \right)^{1/2} \nonumber\\
	 &\leq &C\delta^{{1\over q}+{1\over 2}}  2^{(j_0-j)N}
   2^{jn ({1\over p_0}-{1\over 2})}    \left( \int_X |\varphi_k (\sqrt{L}\,)f(x)|^2  \mathfrak M_{r_0} w(x)d\mu(x)\right)^{1/2}.
  \end{eqnarray*}

  \noindent
  {\bf Estimate for $I\!I(j,k)$.}   Next we show bounds for  the term $ I\!I(j,k).$
For compactly supported function the $L^q$ norm is majorized by the supremum norm, so it follows from ${\rm (ST^{q}_{p_0, 2})}$ that
 \begin{align*} 
 \left\|\bchi_{B_m}
 \left (\psi_{\ell, \delta}\phi_{\delta,j}\right)\left({\sqrt{L}\over t}\right)\right\|_{2\to {p'_0} }
& =\left\| \left ({  \psi}_{\ell, \delta}\phi_{\delta,j}\right) \left({\sqrt{L}\over t}\right) \bchi_{B_m}   \right\|_{p_0\to 2}\nonumber\\
   &\leq C
\big(2^j(1+2^{\ell+2}\delta)\big)^{n ({1\over p_0}-{1\over 2})}\mu(B_m)^{{1\over 2}-{1\over p_0}}
   \|({  \psi}_{\ell, \delta}\phi_{\delta,j})\big((1+2^{\ell+2}\delta) \cdot\big)\|_{\infty}.
         \end{align*}	
From the definition of the function ${ \psi}_{\ell, \delta}$, it follows by \eqref{e4.13}  that,  for $\ell\geq [-{\rm log_2{\delta}}] +1$,
  \begin{eqnarray*}
   \|{  \psi}_{\ell, \delta}\phi_{\delta,j}\big((1+2^{\ell+2}\delta) \cdot\big)\|_{\infty}
   &\leq&
     C_N  2^{j-j_0}\big(2^{j+\ell} \delta\big)^{-N}
	\end{eqnarray*}
for any $N<\infty.$ Therefore,
 \begin{eqnarray}\label{eee}
 I\!I(j,k) &\leq &C
  \sum_{\ell= [-{\rm log_2{\delta}}]+1}^{\infty}   2^{j-j_0}\big(2^{j+\ell} \delta\big)^{n ({1\over p_0}-{1\over 2})-N}
  \left(\sum_m  \mu(B_m)^{{1 }-{2\over p_0}} \|\bchi_{B_m} w\|_{{r_0}}
 \|  \bchi_{\widetilde{B}_m} \varphi_k(\sqrt{L}\,)  f \|^2_{2}\right)^{1/2}\nonumber\\
 &\leq &C  \delta   2^{j[n ({1\over p_0}-{1\over 2})-N+1]} \left(\int_X |\varphi_k (\sqrt{L}\,)f(x)|^2  \mathfrak M_{r_0} w(x)d\mu(x)\right)^{1/2}.
  \end{eqnarray}

 Collecting the estimates of the terms $I(j,k)$ and $ I\!I(j,k)$,  together with  \eqref{e4.16}  and \eqref{e4.18},
we  arrive at the conclusion that
  \begin{eqnarray*}
 \int_X| T_{\delta}f(x)|^2 w(x)d\mu(x)
 &\leq&  C   \delta    \left(  \sum_{j\geq j_0 } \big( \delta^{1\over q} 2^{(j_0-j)N} +2^{-j(N-1)}\big)    2^{nj ({1\over p_0}-{1\over 2 })}\right)^{2}
 \sum_k \int_X |\varphi_k (\sqrt{L}\,)f(x)|^2  \mathfrak M_{r_0} w(x)dx\\
  &\leq&C\delta^{1+{2\over q}}  \delta^{-n(1-{2\over p_0})}       \int_{X} \sum_k| \varphi_k (\sqrt{L}\,)f(x)|^2 \mathfrak M_{r_0} w(x)d\mu(x)\nonumber\\
   &\leq&C\delta^{1+{2\over q}+n(1-{2\over p_0})}  \int_{X}   | f(x)|^2 \mathfrak M_{r_0} w(x)d\mu(x)
  \end{eqnarray*}
   whenever $N> n({1/p_0}-{1/2 }) +1$.  The last inequality follows by Proposition~\ref{prop2.5} for the weighted inequality for the
   square function, since $\mathfrak M_{r_0}w$ is an $A_1$ weight. This proves Lemma~\ref{le4.1} and   completes the proof of Theorem~\ref{th1.1}.
  \end{proof}

\section{Spectral  cluster  estimate and  maximal  bound
}
 \setcounter{equation}{0}

Throughout this section, we assume    that $(X, d, \mu)$ is
a metric measure space satisfying the conditions \eqref{eq1.1} or \eqref{eq2a}.


Let
$1\leq p< 2$ and $2\leq
q\leq\infty$.  Following \cite{COSY}, we say that $L$ satisfies the Sogge spectral cluster condition: If
 for a fixed natural number $\kappa$ and for all $N\in \NN$ and  all even
Borel functions $F$  such that\, $\supp F\subseteq [-N, N]$,
$$
\big\|F(\!\SL\,\,) \big\|_{p\to 2} \leq
CN^{n(\frac{1}{p}-\frac{1}{2})}\| F (N \cdot) \|_{N^\kappa,\, q}
\leqno{\rm (SC^{q, \kappa}_{p})}
$$
for all $x\in X$   where
$$
\|F\|_{N,q}=\left({1\over 2N}\sum_{\ell=1-N}^{N} \sup_{\lambda\in
	[{\ell-1\over N}, {\ell\over N})} |F(\lambda)|^q\right)^{1/q}
$$
for $F$  supported in $[-1, 1]$.  For $q=\infty$, we may put
$\|F\|_{N, \infty}=\|F\|_{{\infty}}$ (see also \cite{CowS, DOS}).

Both {  Theorems B} and {  C} stated in Introduction are a  special case of the following statement with $q=2$.

\begin{thm}\label{th1.2}
	Suppose  that     $L$ satisfies  the  property ${\rm (FS)}$
	and the condition  ${\rm (SC^{q,\kappa}_{p_0})}$ for some $1\leq p_0<2$, $2\leq q\leq \infty $ and  for  some
	$\kappa\in \mathbb N$. In addition,   we assume that there exists $\nu \ge 0$  such that  \eqref{e1.500} holds.
	Then the operator $S_{\ast}^{\alpha}(L)$ is bounded on $L^p(X)$ whenever 
 \begin{eqnarray}\label{eqwww} 	
 2\leq p< p_0',  \ \ \ {\rm and }\ \ \  \alpha> \nu+ 	\max\left\{n\left({\frac {1}{p_0}}-{\frac 12}\right)- {\frac 1q}, \, 0 \right\}. 	
 \end{eqnarray}
 
	As a consequence, if $f\in L^p(X)$, then for $p$ and $\alpha$ satisfying  \eqref{eqwww},
	$$
	\lim\limits_{R\to \infty}S_{R}^{\alpha}(L)f(x)=f(x), \ \ \ a.e.
	$$
	
\end{thm}

\begin{rem}
Note that if $(X,d,\mu)$ satisfies \eqref{eq2a},	 then
by H\"older's inequality and $(1+L)^{0}=Id$, the condition \eqref{e1.500}   holds with $\gamma=0$.
\end{rem}

\begin{rem}
	Taking into account the condition \eqref{st}	one could consider the following estimate introduced in \cite{COSY}:
	$$
	\big\|F(\!\SL\,\,)\bchi_{B(x, r)} \big\|_{p\to 2} \leq
	CV(x,r)^{{1\over s}-{1\over p}}(Nr)^{n({1\over p}-{1\over s})}\| F (N \cdot) \|_{N^\kappa,\, q}.
	$$
	However, one can easily check that the above condition under assumption \eqref{eq1.1} or \eqref{eq2a} is equivalent to ${\rm (SC^{q, \kappa}_{p})}$,
	so here we only discuss the latter only.
\end{rem}

\begin{rem} \label{rem5.4}
Note that condition ${\rm (SC^{q,\kappa}_{p_0})}$ is weaker than ${\rm (ST^{q}_{p_0, 2})}$ and we need  a priori estimate
\eqref{e1.500} in Theorem~\ref{th1.2}. Recall that in \cite[Theorem I.10]{COSY},
one can obtain  $L^p$ bounds for
Bochner-Riesz means  under the assumption  ${\rm (AB_{p_0}^{q,\kappa})}$ instead of estimate \eqref{e1.500}.
Following \cite{COSY}, we say that
  $L$ satisfies the condition  ${\rm (AB_{p_0}^{q,\kappa})}$ if for each $\varepsilon>0$, there exists constant $C_\varepsilon>0$ such that for
all $N\in \mathbb{N}$ and even  Borel functions $H$ with $\supp H\subseteq [-N, N]$,
$$
\big\|H(\SL) \big\|_{p_0\to p_0} \leq
 C_\varepsilon N^{\kappa n({1\over p_0}-{1\over 2})+\varepsilon}\| H (N \cdot) \|_{N^\kappa,\, q}.
\leqno{{\rm (AB_{p_0}^{q,\kappa})}}
$$
(see  also \cite[Theorem 3.6]{CowS} and \cite[Theorem 3.2]{DOS} for related results).
Once  \eqref{e1.500} is proved for some $p_0\in [1, 2 )$ and all $\nu>0$,  it  is not difficult to check  that ${\rm (SC^{q,\kappa}_{p_0})}$ implies ${\rm (AB_{p_0}^{q,\kappa})}$. Indeed,
we apply   \eqref{e1.500} and ${\rm (SC^{q,\kappa}_{p_0})}$ to obtain
\begin{eqnarray*}
\big\|H(\SL) \big\|_{p_0\to p_0}&\leq& \big\|H(\SL)(1+L)^{{n\over 2}({1\over p}-{1\over 2})(\kappa-1)+\varepsilon} \big\|_{p_0\to 2}
\big\| (1+L)^{-{n\over 2}({1\over p_0}-{1\over 2})(\kappa-1)-\varepsilon} \big\|_{2\to p_0}\\
&\leq& C_\varepsilon N^{n({1\over p_0}-{1\over 2})}\big\| H (N \cdot)(1+N^2\cdot)^{{n\over 2}({1\over p_0}-{1\over 2})(\kappa-1)+\varepsilon} \big\|_{N^\kappa,\, q}\\
&\leq& C_\varepsilon  N^{\kappa n({1\over p_0}-{1\over 2})+\varepsilon} \|  H(N \cdot) \|_{N^\kappa,\, q}.
\end{eqnarray*}
This verifies the condition  ${\rm (AB_{p_0}^{q,\kappa})}$.
\end{rem}

From  Remark~\ref{rem5.4}, it is easy to see that  the same  argument as in  \cite[Theorem  I.10, Corollary  I.7]{COSY} gives   the following. 
\begin{prop}  \label{prop5555}
	Under the same assumption as in Theorem \ref{th1.2}, we have the uniform bound 
	$$ 
\big\|\big(I-{L\over R^2}\big)_+^{\alpha}\big\|_{p_0\to p_0}\leq C, 
$$ 
 for  ${\alpha}> \nu+ n(1/p_0-1/2)-1/q$ and $R>0$. 
  As a consequence, if ${1/q}> n({1/p_0}-{1/2}) +\nu$ for some $q\geq 2$ and $1\leq p_0<2$, then  $L=0.$
\end{prop}

This shows that  we may assume the condition
\begin{equation}\label{nontrivial2}
{1\over q}+n\Big({1\over 2}- {1\over p_0}\Big)-\nu\le 0. 
\end{equation}
Because,  otherwise, $L=0$ and Theorem \ref{th1.2} is trivially true. We assume the condition \eqref{nontrivial2} for the rest of this section.   As in Theorem~\ref{th1.1},  Theorem~\ref{th1.2} is a consequence of the following.

\begin{prop}\label{prop5.1}
 Suppose the operator $L$ satisfies the  property {\rm (FS)}
and the condition  ${\rm (SC^{q,\kappa}_{p_0})}$ for some $p_0$ such that  $1\leq p_0<2$, $2\leq q\leq \infty$ and some positive integer $\kappa$.
In addition, assume that there exists $\nu \ge 0$  such that    \eqref{e1.500} holds.
Let $\phi$ be a fixed $C^{\infty}$ function supported in $[-1/2, 1/2], |\phi|\leq 1$.
  Recall that for every $0<\delta\le 1$ $T_\delta$ is defined by (4.6).
Then for all $2\leq p<p'_0 $ and $0<\delta\leq 1$,
 \begin{eqnarray}\label{e5.00}
   \|T_{\delta}f\|_{p}
  &\leq& C(p)\delta^{{1\over 2}+{1\over q}+n({1\over 2}- {1\over p_0})-\nu}   \|f\|_{p}.
  \end{eqnarray}
\end{prop}

 \medskip

The estimate \eqref{e5.00} for $p=2$ follows from  \eqref{einter} and the condition  \eqref{nontrivial2}.  
 To show \eqref{e5.00}  for $2< p<p'_0$, for $0<\delta\leq 1$
we write
\begin{eqnarray}\label{e5.2}
T_\delta f(x)=\left( \int_0^{\infty} \Big|\phi\left(\delta^{-1}\left(1-{ {L}\over t^2} \right) \right)f(x)\Big|^2
   {dt\over t}\right)^{1/2}\leq T_\delta^{(1)}f(x)+T_\delta^{(2)}f(x)+T_\delta^{(3)}f(x),
  \end{eqnarray}
where
\begin{eqnarray*}
T_\delta^{(1)}f(x):&=&\left( \int_0^{1} \Big|\phi\left(\delta^{-1}\left(1-{ {L}\over t^2} \right) \right)f(x)\Big|^2
   {dt\over t}\right)^{1/2},\nonumber\\
T_\delta^{(2)}f(x):&=&\left( \int_1^{1/\sqrt[\kappa]\delta} \Big|\phi\left(\delta^{-1}\left(1-{ {L}\over t^2} \right) \right)f(x)\Big|^2
   {dt\over t}\right)^{1/2},\nonumber\\
T_\delta^{(3)}f(x):&=&\left( \int_{1/\sqrt[\kappa]\delta}^\infty \Big|\phi\left(\delta^{-1}\left(1-{ {L}\over t^2} \right) \right)f(x)\Big|^2
   {dt\over t}\right)^{1/2}.
\end{eqnarray*}

It is clear  that  to prove Proposition~\ref{prop5.1} it is sufficient to show  the following
Lemmas~\ref{le5.1}  and ~\ref{le5.3}.

\begin{lemma} \label{le5.1}
 Suppose the operator $L$ satisfies  the property {\rm (FS)}
and condition  ${\rm (SC^{q,\kappa}_{p_0})}$ for some $p_0$ such that  $1\leq p_0<2$, $2\leq q\leq \infty $ and some $\kappa\in {\mathbb N}^+$.
In addition, we assume that \eqref{e1.500} holds for some $\nu\geq0$.
Then for all $2\leq p\leq p'_0 $ and   $0<\delta\le 1$, we have
\begin{eqnarray}
\label{t1}
\|T_\delta^{(1)}f\|_{{p}}\leq C\delta^{1/2}\|f\|_{{p}}
  \end{eqnarray}
and
\begin{eqnarray}
\label{t2}
\|T_\delta^{(2)}f\|_{{p}}\leq C\delta^{{1\over 2}+{1\over q}+n({1\over 2}- {1\over p_0})-\nu} \|f\|_{{p}}.
  \end{eqnarray}
\end{lemma}

\begin{lemma} \label{le5.3}
 Suppose the operator $L$ satisfies  the  property {\rm (FS)}
and the condition  ${\rm (SC^{q,\kappa}_{p_0})}$ for some $p_0$ such that  $1\leq p_0<2$, $2\leq q\leq \infty $  and some $\kappa\in {\mathbb N}^+$.
Then for all $2\leq p<p'_0 $ and $0<\delta\leq 1$,
 \begin{eqnarray*}
   \|T_{\delta}^{(3)}f\|_{p}
  &\leq& C(p)\delta^{{1\over 2}+{1\over q}+n({1\over 2}- {1\over p_0})}   \|f\|_{p}.
  \end{eqnarray*}
\end{lemma}


\noindent{\bf 5.1. Proof of Lemma~\ref{le5.1}.} From  \eqref{einter},   the proof reduces to showing \eqref{t1} and \eqref{t2} for $p=p'_0$
by interpolation.
By \eqref{e44},  we have that for $f\in L^2(X)\cap L^p(X)$,
 \begin{eqnarray} \label{e5.4}
 | T^{(1)}_{\delta}f(x)|^2
 &\leq&
  5\sum_{k\leq 0}  \int_{2^{k-1}}^{2^{k+2}}  \Big|\phi\left(\delta^{-1}\left(1-{ {L}\over t^2} \right) \right)\varphi_k(\sqrt{L}\,)f(x)\Big|^2
  {dt\over t}.
  \end{eqnarray}
Write
$$
\Phi_{t,\delta}(\sqrt{L}\,):=\phi\left(\delta^{-1}\left(1-{ {L}\over t^2} \right)\right).
$$
Similarly as in Section \ref{sec4} for  $k\in{\mathbb Z}$ and  $\lambda=0, 1, \cdots, \lambda_0=[{8/\delta}] +1$
let $I_\lambda$ and  $\eta_\lambda$ be defined  by \eqref{ijk1} and \eqref{ijk2}, respectively.
Observe that  for every $t\in I_\lambda$,
if $
\Phi_{t,\delta}(s)\eta_{\lambda'} (s)\not=0$ , then  $\lambda-\lambda\delta-3\leq \lambda'\leq \lambda+\lambda\delta+3.$
Hence, we see that,  for every $t\in I_\lambda$,
 \begin{eqnarray}
 \label{k-de}
 \Phi_{t,\delta}(\sqrt{L}\,)
  \varphi_k(\sqrt{L}\,)
 &=&   \sum_{\lambda'=\lambda-10}^{\lambda+10}
 \Phi_{t,\delta}(\sqrt{L}\,)
  \varphi_k(\sqrt{L}\,)
 \eta_{\lambda'}  (\sqrt{L}\,),
  \end{eqnarray}
  and thus
   \begin{eqnarray}
   \label{k-int}
 \int_{2^{k-1}}^{2^{k+2}}
   \Big|\Phi_{t,\delta}(\sqrt{L}\,)
  \varphi_k(\sqrt{L}\,)f\Big|^2    {dt\over t}
  &=& \sum_\lambda\int_{I_\lambda} \Big|\sum_{\lambda'=\lambda-10}^{\lambda+10}
 \Phi_{t,\delta}(\sqrt{L}\,)
  \varphi_k(\sqrt{L}\,)
 \eta_{\lambda'}  (\sqrt{L}\,)f\Big|^2   {dt\over t}\\
 &\leq& C\sum_\lambda\sum_{\lambda'=\lambda-10}^{\lambda+10}\int_{I_\lambda} \Big|
 \Phi_{t,\delta}(\sqrt{L}\,)
  \varphi_k(\sqrt{L}\,)
 \eta_{\lambda'}  (\sqrt{L}\,)f\Big|^2   {dt\over t}.
 \nonumber
  \end{eqnarray}
  %
  By Minkowski's inequality,
  \begin{eqnarray*}
  \|T_\delta^{(1)}f\|_{{p_0'}}&\leq&  C\left\|\left(\sum_{k\leq 0}\int_{2^{k-1}}^{2^{k+2}}
   \Big|\Phi_{t,\delta}(\sqrt{L}\,)
  \varphi_k(\sqrt{L}\,)f\Big|^2    {dt\over t}\right)^{1/2}\right\|_{p_0'}\\
  &\leq& C\left\|\left(\sum_{k\leq 0}\sum_\lambda\sum_{\lambda'=\lambda-10}^{\lambda+10}\int_{I_\lambda}
   \Big|\Phi_{t,\delta}(\sqrt{L}\,)
  \eta_{\lambda'}  (\sqrt{L}\,)\varphi_k(\sqrt{L}\,)f\Big|^2    {dt\over t}\right)^{1/2}\right\|_{p_0'}\\
  &\leq& C\left(\sum_{k\leq 0}\sum_\lambda\sum_{\lambda'=\lambda-10}^{\lambda+10}\int_{I_\lambda}
  \left\|\Phi_{t,\delta}(\sqrt{L}\,)
  \eta_{\lambda'}  (\sqrt{L}\,)\varphi_k(\sqrt{L}\,)f\right\|_{p_0'}^2    {dt\over t}\right)^{1/2}.
  \end{eqnarray*}
 Note that $t\leq 1$ and by ${\rm (SC^{q,\kappa}_{p_0})}$
\begin{eqnarray*}
 \|\Phi_{t,\delta}(\sqrt{L}\,)(1+L)^{\gamma/2}\|_{2\to p_0'}&=& \|\Phi_{t,\delta}(\sqrt{L}\,)(1+L)^{\gamma/2}\|_{p_0\to 2}\\
  &\leq& C2^{n({1\over p_0}-{1\over 2})}\|\Phi_{t,\delta}(2\cdot)\|_{2^{\kappa},q}\leq C.
  \end{eqnarray*}
  Hence  $\big\|
 \Phi_{t,\delta}(\sqrt{L}\,)
  \varphi_k(\sqrt{L}\,)
 \eta_{\lambda'}  (\sqrt{L}\,)f\big\|_{p_0'} \le C \big\|  \eta_{\lambda'}  (\sqrt{L}\,)\varphi_k(\sqrt{L}\,)(1+L)^{-\gamma/2}f\big\|_{2} $.
 From this it is easy to see
   \begin{eqnarray*}
 \|T_\delta^{(1)}f\|_{{p_0'}}
  %
  &\leq&  C\left(\sum_{k\leq 0}\sum_\lambda\sum_{\lambda'=\lambda-10}^{\lambda+10}
  \left\|  \eta_{\lambda'}  (\sqrt{L}\,)\varphi_k(\sqrt{L}\,)(1+L)^{-\gamma/2}f\right\|_{2}^2   \int_{I_\lambda} {dt\over t}\right)^{1/2}\\
  &\leq&  C\left(\delta\sum_{k\leq 0}\sum_{\lambda'}
  \left\|  \eta_{\lambda'}  (\sqrt{L}\,)\varphi_k(\sqrt{L}\,)(1+L)^{-\gamma/2}f\right\|_{2}^2  \right)^{1/2}\\
  &\leq& C\delta^{1\over 2}\|(1+L)^{-\gamma/2}f\|_{2}\\
  &\leq& C\delta^{1\over 2}\|f\|_{{p_0'}},
  \end{eqnarray*}
where for  the last inequality we use \eqref{e1.500}.  Thus we get  \eqref{t1}.

We now show \eqref{t2} for $p=p'_0$.  By \eqref{e44},  we have that for $f\in L^2(X)\cap L^p(X)$,
 \begin{eqnarray} \label{e5.5}
 | T^{(2)}_{\delta}f(x)|^2
 &\leq &
  C\sum_{0<k\leq 1-{\rm log_2 \sqrt[\kappa]{\delta}}}  \int_{2^{k-1}}^{2^{k+2}}
  \Big|\phi\left(\delta^{-1}\left(1-{ {L}\over t^2} \right) \right)\varphi_k(\sqrt{L}\,)f(x)\Big|^2
  {dt\over t}.
  \end{eqnarray}
Again,  for  $k\in{\mathbb Z}$ and  $t\in [2^{k-1}, 2^{k+2}]$ and $\lambda=0, 1, \cdots, \lambda_0=[{8/\delta}] +1$,
we consider the interval $I_\lambda$ and  the function $\eta_\lambda$ which are given by \eqref{ijk1} and \eqref{ijk2}, respectively.
Observe that  for every $t\in I_\lambda$,
if $
\Phi_{t,\delta}(s)\eta_{\lambda'} (s)\not=0$ , then  $\lambda-\lambda\delta-3\leq \lambda'\leq \lambda+\lambda\delta+3.$
Hence, as before it follows that,  for every $t\in I_\lambda$, \eqref{k-de} holds and we have \eqref{k-int}. Putting this in \eqref{e5.5}
and  Minkowski's inequality (twice) give
  \begin{eqnarray*}
  \|T_\delta^{(2)}f\|_{{p_0'}}
  &\leq& C\left\|\left(\sum_{0<k\leq 1-{\rm log_2 \sqrt[\kappa]{\delta}}}\sum_\lambda\sum_{\lambda'=\lambda-10}^{\lambda+10}\int_{I_\lambda}
   \Big|\Phi_{t,\delta}(\sqrt{L}\,)
  \eta_{\lambda'}  (\sqrt{L}\,)\varphi_k(\sqrt{L}\,)f\Big|^2    {dt\over t}\right)^{1/2}\right\|_{p_0'}\\
  &\leq& C\left(\sum_{0<k\leq 1-{\rm log_2 \sqrt[\kappa]{\delta}}}\sum_\lambda\sum_{\lambda'=\lambda-10}^{\lambda+10}\int_{I_\lambda}
  \left\|\Phi_{t,\delta}(\sqrt{L}\,)
  \eta_{\lambda'}  (\sqrt{L}\,)\varphi_k(\sqrt{L}\,)f\right\|_{p_0'}^2    {dt\over t}\right)^{1/2}.
  \end{eqnarray*}
 We claim that
 \begin{equation}
 \label{l-estimate}
\|\Phi_{t,\delta}(\sqrt{L}\,)(1+L)^{\gamma/2}\|_{2\to p_0'}
\leq C\delta^{{1\over q}-n({1\over p_0}-{1\over 2})-\nu}.
\end{equation}
Assuming this for the moment,  we complete the proof.  From \eqref{e5.5} and \eqref{l-estimate} we have
   \begin{eqnarray*}
 \|T_\delta^{(2)}f\|_{{p_0'}}&\leq& C\delta^{{1\over q}-n({1\over p_0}-{1\over 2})-\nu}
 \left(\sum_{0<k\leq 1-{\rm log_2 \sqrt[\kappa]{\delta}}}\sum_\lambda\sum_{\lambda'=\lambda-10}^{\lambda+10}\int_{I_\lambda}
  \left\|  \eta_{\lambda'}  (\sqrt{L}\,)\varphi_k(\sqrt{L}\,)(1+L)^{-\gamma/2}f\right\|_{2}^2    {dt\over t}\right)^{1/2}.
   \end{eqnarray*}
   Thus, it is easy to see that
   \begin{eqnarray*}
 \|T_\delta^{(2)}f\|_{{p_0'}}
  &\leq&  C\delta^{{1\over q}-n({1\over p_0}-{1\over 2})-\nu}\left(\delta\sum_{0<k\leq 1-{\rm log_2 \sqrt[\kappa]{\delta}}}\sum_{\lambda'}
  \left\|  \eta_{\lambda'}  (\sqrt{L}\,)\varphi_k(\sqrt{L}\,)(1+L)^{-\gamma/2}f\right\|_{2}^2  \right)^{1/2}\\
  &\leq& C\delta^{{1\over q}-n({1\over p_0}-{1\over 2})-\nu}\delta^{1/2}\|(1+L)^{-\gamma/2}f\|_{2}\\
&\leq& C\delta^{{1\over 2}+{1\over q}-n({1\over p_0}-{1\over 2})-\nu}\|f\|_{{p_0'}}.
  \end{eqnarray*}
For the last inequality we use  \eqref{e1.500}. This gives the desired estimate.

It remains to show \eqref{l-estimate}.
Let $N=8[t]+1$. Note that $\supp \Phi_{t,\delta}\subset [-N,N]$.
From ${\rm (SC^{q,\kappa}_{p_0})}$
\begin{eqnarray*}
 \|\Phi_{t,\delta}(\sqrt{L}\,)(1+L)^{\gamma/2}\|_{2\to p_0'}&=& \|\Phi_{t,\delta}(\sqrt{L}\,)(1+L)^{\gamma/2})\|_{p_0\to 2}\\
  &\leq& CN^{n({1\over p_0}- {1\over 2})}\|\Phi_{t,\delta}(Nu)(1+N^2u^2)^{\gamma/2}\|_{N^\kappa,q}.
  \end{eqnarray*}
 We estimate $\|\Phi_{t,\delta}(Nu)(1+N^2u^2)^{\gamma/2}\|_{N^\kappa,q}$. Set $H(\lambda)=\Phi_{t,\delta}(\lambda)(1+\lambda^2)^{\gamma/2}$. Let $\xi\in C_c^\infty$ be an even function
 such that $\supp \xi\subset [-1,1], \hat{\xi}(0)=1$ and $\hat{\xi}^{(k)}(0)=0$ for all
$1\leq k\leq [\beta]+2$. Write $\xi_{N}=N\xi(Nu)$. Then
\begin{eqnarray*}
\|\Phi_{t,\delta}(Nu)(1+N^2u^2)^{\gamma/2}\|_{N^\kappa,q}\leq \|\big(H-
\xi_{N^{\kappa-1}}*H\big)(Nu)\|_{N^\kappa,q}
+
\|\big(\xi_{N^{\kappa-1}}*H\big)(Nu)\|_{N^\kappa,q}.
\end{eqnarray*}
To estimate the first in the right hand side, we make use of  the following (for its proof, see \cite[(3.29)]{CowS} or \cite[Propostion 4.6]{DOS}):
If  $\supp G\subset [-1,1]$, then
\begin{eqnarray}\label {eq5.6}
\|G-\xi_N\ast G\|_{N,q}\leq CN^{-\beta}\|G\|_{W^{\beta,q}}
\end{eqnarray}
for all $\beta>1/q$ and any $N\in \NN$.
Since  $\big(H-
\xi_{N^{\kappa-1}}*H\big)(Nu)=H(Nu)-\big(\xi_{N^\kappa}*(H(N\cdot))\big)(u)$.
For $\beta>n(1/2-1/p_0')$,  we get
\begin{align}
\label{first} \big\|\big(H-
\xi_{N^{\kappa-1}}*H\big)(Nu)\big\|_{N^\kappa,q}\leq CN^{-\beta\kappa}\|H(N\cdot)\|_{W^{\beta,q}}&=
CN^{-\beta\kappa}\Big\|\Phi_{t,\delta}(Nu)(1+N^2u^2)^{\gamma/2}\Big\|_{W^{\beta,q}}\\
&\leq CN^{-\beta\kappa+\gamma}\delta^{{1\over q}-\beta}.
\nonumber
\end{align}
For the second one, note that
\begin{eqnarray*}
\|(\xi_{N^{\kappa-1}}*H)(N\cdot)\|_{N^\kappa,q}
   &=&\left(\frac{1}{N^\kappa}\sum_{i=1-N^\kappa}^{N^\kappa}
\sup_{\lambda\in [\frac{i-1}{N^\kappa},\frac{i}{N^\kappa})}|(\xi*H(\cdot/N^{\kappa-1}))(N^\kappa\lambda)|^{q}\right)^{1/q}
\\
   &\leq&\left(\frac{1}{N^\kappa}\sum_{i=1-N^\kappa}^{N^\kappa}
\sup_{\lambda\in [i-1,i)}|(\xi*H(\cdot/N^{\kappa-1}))(\lambda)|^{q}\right)^{1/q}.
\end{eqnarray*}
Using
$
|\xi*h(\lambda)|^q\leq C\|\xi\|^q_{q'}\int_{\lambda-1}^{\lambda+1}|h(u)|^qdu,
$
\begin{eqnarray*}
\|(\xi_{N^{\kappa-1}}*H)(N\cdot)\|_{N^\kappa,q}
   &\leq&C\left(\frac{1}{N^\kappa}\sum_{i=1-N^\kappa}^{N^\kappa}
\sup_{\lambda\in [i-1,i)}\int_{\lambda-1}^{\lambda+1}|H(u/N^{\kappa-1})|^qdu\right)^{1/q}
\\
&\leq&C\left(\frac{1}{N^\kappa}\sum_{i=1-N^\kappa}^{N^\kappa}
\int_{i-2}^{i+1}|H(u/N^{\kappa-1})|^qdu\right)^{1/q}\\
&\leq&CN^{-{\kappa\over q}}\|H(\cdot/N^{\kappa-1})\|_q\\
&\leq& CN^{-{\kappa\over q}}(t^\kappa\delta)^{1\over q}t^{\gamma}\leq C\delta^{1\over q}t^{\gamma}.
\end{eqnarray*}
Combining \eqref{first} and the above, and  noting that $1\leq t\leq 1/\sqrt[\kappa]\delta$, we have
\begin{eqnarray*}
\|\Phi_{t,\delta}(\sqrt{L}\,)(1+L)^{\gamma/2}\|_{2\to p_0'}&\leq&
 CN^{n({1\over p_0} -{1\over 2})}(N^{-\beta\kappa+\gamma}\delta^{{1\over q}-\beta}+t^{\gamma}\delta^{1\over q})\\
&\leq& C\delta^{{1\over q}-n({1\over p_0}-{1\over 2})-\nu}.
\end{eqnarray*}
Here, we use  the relation $\gamma=n(\kappa-1)(1/p_0-1/2)+\kappa\nu$.  This gives \eqref{l-estimate}, and completes the proof of \eqref{t2}.
 \hfill{}$\Box$


\noindent{\bf 5.2. Proof of Lemma~\ref{le5.3}.}\ As in Proposition~\ref{prop4.1}, the proof of
  Lemma~\ref{le5.3} reduces to showing the following lemma.

\begin{lemma} \label{le5.5}
For any $0\leq w$ and $0<\delta\le 1$,
\begin{eqnarray*}
\int_X |T_\delta^{(3)} f(x)|^2 w(x) d\mu(x)  \leq C\delta^{1+{2\over q}+n({  2\over p'_0} -1)} \int_X |f(x)|^2 \mathfrak M_{r_0}w(x)d\mu(x),
  \end{eqnarray*}
where $1/r_0 +2/p'_0=1.$
\end{lemma}

\begin{proof}
We prove  Lemma~\ref{le5.5}  by modifying  that of Lemma~\ref{le4.1}.
By \eqref{e44},  we have that, for  $f\in L^2(X)\cap L^p(X)$,
 \begin{eqnarray} \label{e5.8}
 | T_{\delta}^{(3)}f(x)|^2
 &\leq &
  C\sum_{k>1-{\rm log_2 \sqrt[\kappa]{\delta}}} \int_{2^{k-1}}^{2^{k+2}}
  \Big|\phi\left(\delta^{-1}\left(1-{ {L}\over t^2} \right)\right)\varphi_k(\sqrt{L}\,)f(x)\Big|^2
  {dt\over t}.
  \end{eqnarray}

 For  given  $0<\delta\leq 1$,  we let
$\delta\in [2^{-j_0-1}, 2^{-j_0})$ for some $j_0\in{\mathbb Z}$.  As in
the proof Lemma \ref{le4.1} we fix  a cutoff function  $\eta\in C_0^{\infty}$, identically one on
$\{|s| \leq 1 \}$  and supported on $\{|s| \leq 2 \}$.  For $j\ge j_0$ we define  $\zeta_j$ by \eqref{zeta} so that \eqref{id} holds. Then let $
  \phi_{\delta,j}$ be defined by \eqref{e4.12}  so that  \eqref{e4.14} holds.
From \eqref{e5.8} and \eqref{e4.14},  it follows  that for every function $w\geq 0,$
 \begin{eqnarray}\label{e5.11}\hspace{0.8cm}
  \int_X | T_{\delta}^{(3)}f(x)|^2 w(x) d\mu(x)
  &\leq&
  C\sum_{k>1-{\rm log_2 \sqrt[\kappa]{\delta}}}\left[ \sum_{j\geq j_0}   \left(\int_{2^{k-1}}^{2^{k+2}}
  \left\langle  \Big|\phi_{\delta,j}\left({\sqrt{L}\over t} \right)
  \varphi_k(\sqrt{L}\,)f\Big|^2, \  w\right\rangle   {dt\over t} \right)^{1/2}\right]^{2}.
 \end{eqnarray}
For $\ell\ge 0$ let ${ \psi}_{\ell, \delta}$ be defined by \eqref{psi}. So,  $1=\sum_{\ell=0}^{\infty} { \psi}_{\ell, \delta}(s)$, and so
 $
  \phi_{\delta,j}(s)=  \sum_{\ell=0}^{\infty}  \big ( { \psi}_{\ell, \delta}\phi_{\delta,j}\big)(s)
  $ for all $s>0.
 $
Similarly as in \eqref{e4.18}   we get
   \begin{eqnarray}\label{e5.12}\hspace{0.5cm}
 &&\hspace{-1.2cm} \left(\int_{2^{k-1}}^{2^{k+2}}
  \left\langle  \Big| \phi_{\delta,j}\left({\sqrt{L}\over t} \right)
  \varphi_k(\sqrt{L}\,)  f\Big|^2, \  w\right\rangle   {dt\over t} \right)^{1/2}\nonumber\\
 &\leq & \sum_{\ell=0}^{ [-{\rm log_2{\delta}}]}
  \left(   \sum_{m }   \|\bchi_{B_m}w\|_{{r_0}} \int_{2^{k-1}}^{2^{k+2}}  \left\|\bchi_{B_m}
 \left ( { \psi}_{\ell, \delta}\phi_{\delta,j}\right)\left({\sqrt{L}\over t}\right)
 \bchi_{\widetilde{B}_m} \varphi_k(\sqrt{L}\,)  f \right\|^2_{{p'_0}} {dt\over t}\right)^{1/2}
\nonumber\\
 &+&
  \sum_{\ell= [-{\rm log_2{\delta}}]+1}^{\infty}
 \left(  \sum_{m }   \|\bchi_{B_m} w\|_{{r_0}}  \int_{2^{k-1}}^{2^{k+2}}  \left\|\bchi_{B_m}
 \left ( { \psi}_{\ell, \delta}\phi_{\delta,j}\right)\left({\sqrt{L}\over t}\right)
 \bchi_{\widetilde{B}_m} \varphi_k(\sqrt{L}\,)  f \right\|^2_{{p'_0}} {dt\over t}\right)^{1/2}\nonumber\\
   &=&   I(j,k) + I\!I(j,k).
  \end{eqnarray}
 As in Section \ref{sec4}, the first term
$I(j,k)$ is the major one. We handle $I\!I(j,k)$ first.

  {\bf Estimates for $I\!I(j,k)$}.
Note that $\|F\|_{N,2} \le \|F\|_{N,\infty}=\|F\|_{\infty}$,
so  for a fixed $b>0$ the condition  ${\rm (SC^{q, \kappa}_{p_0})}$ implies $ {\rm (ST^{\infty}_{p_0, 2})}$
for all functions $F$ with $\supp F \subset (b,R)$. Hence we can repeat the same argument used for  the proof of \eqref{eee}
to show that for any $N<\infty$
 $$
 I\!I(j,k)
 \leq  C  \delta   2^{j[n ({1\over p_0}-{1\over 2})-N+1]} \left(\int_X |\varphi_k (\sqrt{L}\,)f(x)|^2  \mathfrak M_{r_0} w(x)d\mu(x)\right)^{1/2}.
 $$
{\bf Estimates for $I(j,k)$}.
As before (see Section \ref{sec4} ), for  $k\in{\mathbb Z}$ and  $t\in [2^{k-1}, 2^{k+2}]$ and $\lambda=0, 1, \cdots, \lambda_0=[{8/\delta}] +1$,
we consider the interval $I_\lambda$ and  the function $\eta_\lambda$ which are given by \eqref{ijk1} and \eqref{ijk2}, respectively.
For $t\in I_\lambda$,
$\lambda-2^{\ell+6}\leq \lambda'\leq \lambda+2^{\ell+6}$ if $
{\psi}_{\ell, \delta} \left({s/t}\right)\eta_{\lambda'} (s)\not=0$.
Thus,  for  $t\in I_\lambda$,  we have \eqref{eqn4}. Using this we get
 \begin{eqnarray*}
 I(j,k)\leq  \sum_{\ell=0}^{ [-{\rm log_2{\delta}}]} \left[ \sum_{m }  \|\bchi_{B_m} w\|_{{r_0}}
 \sum_{\lambda}  \int_{I_\lambda}  \left(\sum_{\lambda'=\lambda-2^{\ell+6}}^{\lambda+2^{\ell+6}}  \left\|\bchi_{B_m}
 \left ({ \psi}_{\ell, \delta}\phi_{\delta,j}\right)\left({\sqrt{L}\over t}\right)
 \eta_{\lambda'}  (\sqrt{L}\,)\big[\bchi_{\widetilde{B}_m} \varphi_k(\sqrt{L}\,)  f\big] \right\|_{{p'_0}}\right)^2  {dt\over t}\right]^{1/ 2}.
  \end{eqnarray*}

Note  that  $ \supp { \psi}_{\ell, \delta} \subseteq (1-2^{\ell+2}\delta, 1+2^{\ell+2}\delta).$
 Moreover,  if $\ell\geq 1$,   then  ${ \psi}_{\ell, \delta}(s)=0$  for $s\in (1-2^{\ell}\delta, 1+2^{\ell}\delta)$,
and so $\supp \left (\psi_{\ell, \delta}\phi_{\delta,j}\right)\left({\cdot/t}\right)\subset [t(1-2^{\ell+2}\delta), t(1+2^{\ell+2}\delta)]$.
Let $R=[t(1+2^{\ell+2}\delta)]+1$.
By the  condition   ${\rm (SC^{q,\kappa}_{p_0})}$, we have that,  for $0\leq \ell\leq [-{\rm log_2{\delta}}]$,
 \begin{eqnarray} \label{e5.13} \hspace{1cm}
 \left\|\bchi_{B_m}
 \left (\psi_{\ell, \delta}\phi_{\delta,j}\right)\left({\sqrt{L}\over t}\right)\right\|_{2\to {p'_0} }&=&
\left\| \left ({  \psi}_{\ell, \delta}\phi_{\delta,j}\right) \left({\sqrt{L}\over t}\right) \bchi_{B_m}   \right\|_{p_0\to 2}\nonumber\\
   &\leq& C
   \left(2^j(1+2^{\ell+2}\delta)\right)^{n ({1\over p_0}-{1\over 2})}\mu(B_m)^{{1\over 2}-{1\over p_0}}
   \|({  \psi}_{\ell, \delta}\phi_{\delta,j})\big(R \cdot/t\big)\|_{R^\kappa,q}\nonumber\\
   &\leq& C
 2^{jn ({1\over p_0}-{1\over 2})}\mu(B_m)^{{1\over 2}-{1\over p_0}}
   \|({  \psi}_{\ell, \delta}\phi_{\delta,j})\big(R \cdot/t\big)\|_{R^\kappa,q}.
         \end{eqnarray}


We note that
$$
\supp \,({  \psi}_{\ell, \delta}\phi_{\delta,j})\big(R \cdot/t\big)\subset \left[\frac{t(1-2^{\ell+2}\delta)}{R},\, \frac{t(1+2^{\ell+2}\delta)}{R}\right].
$$
This, in combination with  the fact that $R^\kappa\delta\geq 1$,  gives
\begin{eqnarray}
\label{e5.14}
 \|({  \psi}_{\ell, \delta}\phi_{\delta,j})\big((1+2^{\ell+2}\delta) \cdot\big)\|_{R^\kappa,q}
\leq  \|{  \psi}_{\ell, \delta}\phi_{\delta,j}\|_\infty \big\|\chi_{[\frac{t(1-2^{\ell+2}\delta)}{R},\, \frac{t(1+2^{\ell+2}\delta)}{R}]}\big\|_{R^\kappa,q}
\le C \|{  \psi}_{\ell, \delta}\phi_{\delta,j}\|_\infty  \left(\frac{2^{\ell+3}t\delta}{R}\right)^{1/q}.\nonumber
 \end{eqnarray}
 From this and \eqref{psidel} with $q=\infty$ we see that
  \begin{eqnarray*}
 \|({  \psi}_{\ell, \delta}\phi_{\delta,j})\big((1+2^{\ell+2}\delta) \cdot\big)\|_{R^\kappa,q}
&\leq  &C_N 2^{(j_0-j)N} 2^{- \ell N } (2^\ell\delta)^{1\over q}.
\end{eqnarray*}
Thus \eqref{e5.13} and the above inequality yield
	\begin{eqnarray} \label{e5.15}
 &&\hspace{-1.8cm}   \left\|\bchi_{B_m}
 \left ({ \psi}_{\ell, \delta}\phi_{\delta,j}\right)\left({\sqrt{L}\over t}\right)
 \eta_{\lambda'}  (\sqrt{L}\,)\big[\bchi_{\widetilde{B}_m} \varphi_k(\sqrt{L}\,)  f\big] \right\|_{{p'_0}}\\
 &\leq&
 C\delta^{1\over q}  2^{(j_0-j)N}
   2^{jn ({1\over p_0}-{1\over 2})} 2^{- \ell( N-{1\over q})} \mu(B_m)^{{1\over 2}-{1\over p_0}}   \big\|
 \eta_{\lambda'}  (\sqrt{L}\,)\big[\bchi_{\widetilde{B}_m} \varphi_k(\sqrt{L}\,)  f\big] \big\|_{2}. \nonumber
  \end{eqnarray}
Once \eqref{e5.15} is obtained,  we may repeat  the lines of argument in the proof of Lemma~\ref{le4.1} to get
 \begin{eqnarray*}
 I(j,k)
 &\leq &  C\delta^{{1\over q}+{1\over 2}}  2^{(j_0-j)N}
   2^{jn ({1\over p_0}-{1\over 2})}    \int_X |\varphi_k (\sqrt{L}\,)f(x)|^2  \mathfrak M_{r_0} w(x)d\mu(x).
  \end{eqnarray*}

Finally, combining the estimates for  $I(j,k)$ and $ I\!I(j,k)$,  together with  \eqref{e5.11}  and \eqref{e5.12},
we  get
  \begin{eqnarray*}
 \int_X| T_{\delta}^{(3)}f(x)|^2 w(x)d\mu(x)
 &\leq&  C\delta^{1+{2\over q}+n(1-{2\over p_0})}  \int_{X}   | f|^2 \mathfrak M_{r_0} w(x)dx
  \end{eqnarray*}
   whenever $N> n({1/p_0}-{1/2 }) +1$.
  This completes proof of Lemma ~\ref{le5.5}.
 \end{proof}

 \medskip

\section{Applications }\label{sec6}
\setcounter{equation}{0}

 As applications of our theorems  we discuss several examples of important elliptic operators.
Our results,  Theorems~\ref{th1.1}and ~\ref{th1.2}
  have applications to  all the examples which are
  discussed in \cite{DOS} and \cite{COSY}. Those  include
elliptic operators on compact manifolds,   the harmonic oscillator,
 radial  Schr\"odinger operators with inverse square potentials
 and the Schr\"odinger operators  on asymptotically conic manifolds.


\noindent{\bf 6.1.  Laplace-Beltrami operator on compact manifolds.}\
 Let  $\Delta_g$ be  the Laplace-Beltrami operator on a compact smooth
 Riemannian manifold $(M,g)$ of dimension $n$.
 It was shown by Sogge that  the condition  ${\rm(S_{p})}$ holds with $L=-\Delta_g$
in the standard range of Stein-Tomas restriction theorem, that is to say, for $1\leq p\leq 2(n+1)/(n+3)$, see \cite{Sog1, Sog3}. Hence we can apply
Theorem~\ref{th1.2} and obtain the following.

 \begin{cor}\label{prop6.1}  Suppose that $\Delta_g$ is the Laplace-Beltrami operator on a compact smooth
 	Riemannian manifolds $(M,g)$ of dimension $n$.
Then   the operator $S_{\ast}^{\alpha}(-\Delta_g)$ is bounded on $L^p(M)$ whenever
\begin{eqnarray}\label{e6.1}
  p \ge {2(n+1)\over n-1}, \ \ \ {\rm and}\ \   \alpha>
  \max\left\{  n\Big|{1\over p}-{\frac 1 2}\Big|- {\frac 1 2}, \, 0 \right\}.
\end{eqnarray}

 \end{cor}



 The  corollary can be extended to  the Laplace-Beltrami operator on a certain class of  compact manifolds with
 boundaries if one combines Theorem \ref{th1.2} and the results in Sogge \cite{Sog4}.
As far as we are aware,  Corollary \ref{prop6.1}, especially in view of its generality, has not appeared  in  any
 literature before. 
 However, we should  mention that in \cite{MSS} Mockenhaupt, Seeger and Sogge showed that  the sharp maximal
 Bochner-Riesz bounds for $p\ge 2$  holds when  $(M, g)$
 is a compact Riemannian manifold of dimension $2$ with  periodicity assumption for  the geodesic flow.


\noindent{\bf 6.2. Schr\"odinger operator on asymptotically conic manifolds.}\
Scattering manifolds  or  asymptotically conic manifolds
 are defined as  interiors of a compact manifold with boundary $M$, and the metric $g$
 is smooth on the interior $M^\circ$ and has the form
\[g=\frac{dx^2}{x^4}+\frac{h(x)}{x^2} \]
in a collar neighbourhood near $\partial M$, where $x$ is a smooth boundary defining function
for $M$ and $h(x)$ a smooth one-parameter family of metrics on $\partial M$; the function $r:=1/x$
near $x=0$ can be thought of as a radial coordinate near infinity and the metric there is
asymptotic to the exact metric cone $((0,\infty)_r \times \partial M, dr^2+r^2h(0))$.

%
%

The restriction estimate \eqref{e1.4} and Bochner-Reisz sumability results for a class of Laplace type operators on
on asymptotically conic manifolds were obtained in \cite{GHS}.
Our approach allows us to complement these results  with the following concerning the maximal Bochner-Riesz operator.

\begin{cor}\label{prop8.1} Let $(M,g)$ be an
asymptotically conic nontrapping
manifold of dimension $n \geq 3$, and let $x$ be a smooth boundary defining function of
$\partial M$. Let $L:= - \Delta_g+V$ be a Schr\"odinger operator
with $V\in x^3C^\infty(M)$ and assume that $L$ has no $L^2$-eigenvalues and that $0$ is not a resonance.
Then  the operator $S_{\ast}^{\alpha}(L)$ is bounded on $L^p(M)$ whenever
\begin{eqnarray*}
p \ge {2(n+1)\over n-1}, \ \ \ {\rm and}\ \   \alpha>
\max\left\{  n\left|{1\over p}-{\frac 1 2}\right|- {\frac 1 2}, \, 0 \right\}.
\end{eqnarray*}
\end{cor}

\begin{proof} Corollary~\ref{prop8.1} follows  from restriction estimates \eqref{e1.4} established in \cite[Theorem 1.2]{GHS}  and  Theorem~A.
\end{proof}

Corollary~\ref{prop8.1} includes  a class of operators which are 0-th order perturbations
of the Laplacian on nontrapping asymptotically conic manifolds.
In particular, our results cover the following settings:
  the Schr\"odinger operators, i.e. $-\Delta + V$ on $\RR^n$, where $V$ smooth and decaying sufficiently at infinity;
the Laplacian with respect to  metric perturbations of
the flat metric on $\RR^n$, again decaying sufficiently at infinity;
 and the Laplacian on asymptotically conic manifolds, see \cite{GHS}.


\noindent{\bf 6.3. The harmonic oscillator.}\ In this section we focus  on the Schr\"odinger operators such as the  harmonic oscillator
$-\Delta + |x|^2$ on $L^2(\RN)$
for $n\ge 2$. Bochner-Riesz  summability results for the harmonic oscillator were  studied and
sharp results were obtained  by Karadzhov \cite{Kar} and Thangavelu in \cite{Th3, Th4}). Here we establish the
corresponding result for the maximal Bochner-Riesz operator.
However, we  consider the  class
Schr\"odinger operators $
L= - \Delta + V(x)$
with a  positive potential $V$ which satisfies the following condition
\begin{equation}\label{eq111.01}
V(x) \sim |x|^2, \quad |\nabla V(x)| \sim |x|, \quad |\partial_x^2 V(x)| \lesssim 1.
\end{equation}
Clearly this class includes the harmonic oscillator.

 A restriction type result for this class of operators was established by Koch
 and Tataru in \cite[Theorem 4]{KoT}, which states  that,  for $\lambda \ge 0$  and $1\leq p\leq 2n/(n+2)$,
\begin{eqnarray*}
\|E_L[\lambda^2,\lambda^2+1)\|_{p\to 2} \leq C(1+\lambda)^{n({1\over p}-{1\over 2})-1}.
\end{eqnarray*}
 It is not difficult to show that the above condition is equivalent to
  condition  ${\rm (SC^{2,\kappa}_{p})}$ for $\kappa=2$ and $1\leq p\leq 2n/(n+2)$, see \cite{COSY}.

As a consequence of Theorem~\ref{th1.2} we establish boundedness of the associated maximal Bochner-Riesz operator.

  \begin{cor}\label{cor6.3} Let $L= - \Delta + V(x)$
with a  positive potential $V(x)$ satisfying \eqref{eq111.01}.
Then   the operator $S_{\ast}^{\alpha}(L)$ is bounded on $L^p(\RN)$ whenever
\begin{eqnarray}\label{e6.3}
 p \ge {2n\over n-2}\ \ \ {\rm and }\ \ \  \alpha>
  \max\left\{n\Big|{1\over p}-{\frac 1 2}\Big|- {\frac 1 2}, \, 0 \right\}.
\end{eqnarray}
 \end{cor}

\begin{proof}

 As we just mentioned,  the condition   ${\rm (SC^{2,\kappa}_{p'})}$ for $\kappa=2$ and  $1\leq p'\leq \frac{2n}{n+2}$ follows from
 \cite[Thoerem 4]{KoT} and \cite[Theorem I\!II.9]{COSY}.  Hence by { Theorem~C} it is enough to show that
if $V(x)\sim |x|^2$ is a  positive potential  and     $L= - \Delta + V$, then
 \begin{eqnarray}\label{e77.111}
 \|(1+L)^{-\gamma/2}\|_{2 \to p'}\leq C, \,\, \gamma=n(1/p-1/2)+2\nu
 \end{eqnarray}
for   $1\leq p'\leq \frac{2n}{n+2}$ and  all $\nu>0$.
 The proof of \eqref{e77.111} for $p=1$ is given in  \cite[Lemma 7.9]{DOS}. We give a brief proof of this for completeness.

 Now fix $\nu$ as a positive number. To prove \eqref{e77.111}, we put  $M=M_{\sqrt{1+V}}$. Then we note that
 $$
 \|(1+L)^{1/2}f\|_2^2=\langle(1+L)f,f\rangle \geq \langle M^2f,f\rangle=\|Mf\|_2^2.
 $$
By the L\"owner-Heinz inequality for any quadratic forms $B_1$ and $B_2$, if $B_1\geq B_2\geq 0$, then $B_1^\alpha\geq B_2^\alpha$ for
 $0\leq \alpha\leq 1$. Hence, $$\langle(1+L)^\alpha f,f\rangle \geq \langle M^{2\alpha}f,f\rangle.$$ Thus, for $\alpha\in [0,1]$,
 \begin{eqnarray}\label{e6.10a}
 \|M^\alpha(1+L)^{-\alpha/2}\|_{2\to 2}\leq C .
 \end{eqnarray}
 For $\alpha=1$ the operator $M^\alpha(1+L)^{-\alpha/2}$ is of a first order Riesz transform type and a standard argument yields,
 for any $q\in (1,2]$,
\begin{eqnarray}\label{e6.10}
 \|M(1+L)^{-1/2}\|_{q\to q}\leq C,
\end{eqnarray}
 see \cite[Theorem 11]{S2}. Then  by H\"older's inequality,
 for any $q_1\geq q_2\geq 1$ with $s=(1/q_2-1/q_1)^{-1}$,
\begin{eqnarray}\label{e6.11}
 \|M^{-\alpha}\|_{q_1\to q_2}\leq C \left(\int_{{\mathbb R}^n} (1+V(x))^{-s\alpha/2}dx\right)^{1/(s\alpha)}.
\end{eqnarray}
Recall that $\gamma=n(1/p-1/2)+2\nu$.  Write
 \begin{eqnarray}\label{ppp}
(1+L)^{-\gamma/2}=\big(M^{-1}M(1+L)^{-1/2}\big)^{[\gamma]} M^{[\gamma]-\gamma} M^{\gamma-[\gamma]}(1+L)^{([\gamma]-\gamma)/2}.
\end{eqnarray}
Because of $V(x)\sim |x|^2$, choose
$s=(n+\varepsilon)/\alpha$ in \eqref{e6.11} with
 $\varepsilon= 2\nu/({1/p'}-{1/2})>0. $
Denote $p_0 $ by
 ${1/p_0}=
{(\gamma-[\gamma])/(n+\varepsilon)} +{1/2}
 $
and for each $1\leq i\leq [\gamma]-1$ we define $p_i$ by putting $1/p_{i+1}- 1/p_i=1/(n+\varepsilon)$, so $p_{[\gamma]}=p'$.
Now multiple composition of operators from  \eqref{e6.10a}, \eqref{e6.10}  and \eqref{e6.11}, in combination with \eqref{ppp},  yield
\begin{eqnarray*}
 \|(1+L)^{-\gamma/2}\|_{2 \to p'}
 \leq \|M^{\gamma-[\gamma]}(1+L)^{([\gamma]-\gamma)/2}\|_{2\to 2} \|M^{[\gamma]-\gamma}\|_{2\to p_0}
\prod_{i=0}^{[\gamma]-1}\|M^{-1}M(1+L)^{-1/2}\|_{p_i\to p_{i+1}}
 \leq  C.
\end{eqnarray*}
This finishes the proof of \eqref{e77.111}, and completes  the proof of Corollary~\ref{cor6.3}.
 \end{proof}


\noindent{\bf 6.4. Operators $\Delta_n+  {c\over r^{2}}$ acting on $L^2((0,\infty), r^{n-1}dr)$.}\
In this section  we consider a class of the Schr\"odinger operators
on $L^2((0,\infty), r^{n-1} dr)$. These operators
generate semigroups but do not  have the classical Gaussian upper bound for the  heat kernel.

Fix  $n > 2$ and   $ c>  -{(n-2)^2/4} $ and consider the space $L^2((0,\infty), r^{n-1}dr)$.
For  $f,g \in C_c^\infty(0,\infty)$
we define the quadratic form
\begin{equation}
Q_{n,c}^{(0,\infty)}  (f,g)=\int^{\infty}_{0}f'(r)g'(r)r^{n-1} dr
+\int_0^\infty  \frac{c}{r^2} f(r)g(r) r^{n-1}dr.
\label{eq9.1}
\end{equation}
Using the Friedrichs extension one can define the operator  $L_{n,c}=
\Delta_n+c/{r^2}$ as the unique self-adjoint operator corresponding
to $ Q_{n,c}^{(0,\infty)} $, acting on $L^2((0,\infty), r^{n-1}dr)$. In the sequel we will write
$L$ instead of $L_{n,c}$, which is formally given by the following formula
$$
Lf=(\Delta_n+\frac{c}{r^2} ) f=-\frac{d^2}{dr^2}f-\frac{n-1}{r}\frac{d}{dr}f +\frac{c}{r^2}f.
$$
 The classical Hardy  inequality
\begin{equation}\label{e6.99}
- \Delta\geq  {(n-2)^2\over 4}|x|^{-2},
\end{equation}
shows that  for all $c > -{(n-2)^2/4}$,  the self-adjoint operator  $L$
is non-negative. Such operators can be seen as radial  Schr\"odinger operators
with inverse-square potentials.
 It follows by Theorem 3.3 of \cite{CouS} that
$L$ satisfies the property (FS).

Now for $   -{(n-2)^2/4}<c<0,$ we set $p_c^{\ast}=n/\sigma$  where $\sigma=(n-2)/2-\sqrt{(n-2)^2/4+c}$.   Note that $2<{2n\over n-2}<p_c^{\ast}$.
Liskevich, Sobol
and Vogt \cite{LSV}  proved that, for $t>0$ and
$p\in ((p_c^{\ast}){'}, p_c^{\ast})$,
$$
\|e^{-tL}\|_{p\to p}\leq C.
$$
They also proved that the range of $p$,  $((p_c^{\ast}){'}, p_c^{\ast})$ is optimal in  the sense that,
if $p\not\in ((p_c^{\ast}){'}, p_c^{\ast})$, the semigroup does not act on
$L^p((0,\infty), r^{n-1}dr)$
(see also  \cite{CouS}).

\begin{cor}\label{prop8.5} Suppose that $n > 2$  and $  -{(n-2)^2/4}<c $. Set
\begin{eqnarray*}
  p_c^{\ast}=
  \left\{
  \begin{array}{ll}
  {n\over \sigma}, \ \ \  c<0;\\[6pt]
  \infty, \ \ \ \ c\geq 0,
  \end{array}
  \right.
  \end{eqnarray*}
 where $\sigma=(n-2)/2-\sqrt{(n-2)^2/4+c}$.
Then   the operator $S_{\ast}^{\alpha}(L)$ is bounded on $L^p((0,\infty), r^{n-1}dr)$ whenever
\begin{eqnarray}\label{e6.9}
{2n\over n-1} <p<  p_c^{\ast} \ \ \ {\rm and }\ \ \  \alpha>
  \max\left\{  n\left|{1\over p}-{\frac 1 2}\right|- {\frac 1 2}, \, 0 \right\}.
\end{eqnarray}

 \end{cor}
\begin{proof}
It was shown in \cite[Proposition I\!I\!I.10]{COSY} that the
condition ${\rm (ST^{2}_{p, 2})}$ for the operators $\Delta_n+c/{r^2}$ holds
for
 $p\in \big((p_c^{\ast}){'}, {2n\over n+1}\big)$ for $c<0$;  for $p\in \big[1, {2n\over n+1}\big)$
 for $c\ge 0$. Now the corollary follows from Theorem \ref{th1.1}.
\end{proof}

\begin{rem}
In the proof of Corollary \ref{prop8.5} one has to use condition  ${\rm (ST^{2}_{p, 2})}$
because the condition \eqref{rp} is no longer valid in this setting.
\end{rem}

Finally we mention that
our approach can be also applied to a class of sub-Laplacians on Heisenberg $H$-type group considered in \cite{LW},
for the class of inverse square potentials considered
in \cite{BPST} and a class of Schr\"odinger type operators investigated in \cite{RS}.

\noindent
{\bf Acknowledgements:}
P. Chen was supported by NNSF of China 11501583, Guangdong Natural Science Foundation 2016A030313351 and the Fundamental Research Funds for the Central Universities 161gpy45. S. Lee  was partially supported by NRF (Republic of Korea) grant No. 2015R1A2A2A05000956.
A. Sikora was partly  supported by
Australian Research Council  Discovery Grant DP DP160100941.
 L.X. Yan was supported by the NNSF
of China, Grant No. ~11471338 and  ~11521101, and Guangdong Special Support Program.

\end{document}